\documentclass[11pt]{amsart}
\usepackage{fullpage}
\usepackage{amsfonts,amssymb,amsmath,amsthm}
\usepackage{hyperref}

\theoremstyle{plain}
\newtheorem{theorem}{Theorem}[section]
 \newtheorem{corollary}[theorem]{Corollary}
 \newtheorem{lemma}[theorem]{Lemma}
 \newtheorem{proposition}[theorem]{Proposition}
 \theoremstyle{definition}
 \newtheorem{definition}[theorem]{Definition}
 
 \theoremstyle{remark}
 \newtheorem{remark}[theorem]{Remark}
 \numberwithin{equation}{section}
  \newtheorem{example}[theorem]{Example}

\def\eps{\varepsilon}
\def\R{{\mathbb R}}
\def\N{{\mathbb N}}
\mathchardef\mhyphen="2D 
\def \PD { {\rm P\!\mhyphen \!\rm D} } 
\def \pd { {{\mathfrak p}\mhyphen {\mathfrak d}}}
\def \PM { {\rm P\!\mhyphen \!\rm M} } 
\def \ad {{\rm ad}}
\def \Ghat{\widehat{G}}
\def \Hhat{\widehat{H}}
\def \PullOne{\mathcal{I}_\Phi}
\def \PullTwo{\mathcal{U}_\Phi}
\def \pr {{\rm pr}} 

\def\Tend#1#2{\mathop{\longrightarrow}\limits_{#1\rightarrow#2}}

\numberwithin{equation}{section}

\begin{document}
\title[]{Geometric invariance of the semi-classical calculus\\ 
on nilpotent graded Lie groups}
\author[C. Fermanian]{Clotilde~Fermanian~Kammerer}
\address[C. Fermanian Kammerer]{
Univ Paris Est Creteil, CNRS, LAMA, F-94010 Creteil, Univ Gustave Eiffel, LAMA, F-77447 Marne-la-Vallée, France
}
\email{clotilde.fermanian@u-pec.fr}
\author[V. Fischer]{V\'eronique Fischer}\address[V. Fischer]%
{University of Bath, Department of Mathematical Sciences, Bath, BA2 7AY, UK} 
\email{v.c.m.fischer@bath.ac.uk}
\author[S. Flynn]{Steven Flynn}
\address[S. Flynn]{University of Bath, Department of Mathematical Sciences, Bath, BA2 7AY, UK} 
\email{spf34@bath.ac.uk}

\begin{abstract} 
In this paper, we consider the semi-classical setting constructed on nilpotent graded Lie groups by means of representation theory.
Our aim is to analyze 
the effects of the pull-back by diffeomorphisms on pseudodifferential operators. We restrict to diffeomorphisms  that preserve the filtration  and prove that they are uniformly
Pansu differentiable.  We show that
 the pull-back of a semi-classical pseudodifferential operator by such a diffeomorphism  has a semi-classical symbol that is  expressed at leading order in terms of the Pansu differential. 
 Finally, we interpret the geometric meaning of this invariance  in the setting of filtered manifolds. 
 \end{abstract}

\subjclass[2010]{43A80, 47G30, 58J40}
\keywords{
Analysis on nilpotent Lie groups, 
Semi-classical analysis on nilpotent Lie groups and on filtered manifolds,
Abstract harmonic analysis.  
}

\maketitle

\makeatletter
\renewcommand\l@subsection{\@tocline{2}{0pt}{3pc}{5pc}{}}
\makeatother

\tableofcontents

\section{Introduction}

Semi-classical analysis started to develop in the 70s and has proved a flexible tools for the analysis of PDEs~\cite{Helf}. It crucially relies on a microlocal viewpoint and the use of semi-classical pseudodifferential calculus. Such an approach has been recently developed on nilpotent  graded Lie groups with the ambition of taking into account the non-commutativity of the group by using  representation theory and the associated Fourier theory~\cite{BFG,FR,FF1}. The symbols of the pseudodifferential operators are then fields of operators on the product $G\times\widehat G$ of the group $G$ and its dual space $\widehat G$.
 Like in the Euclidean case~\cite{Zwobook}, this  semi-classical pseudodifferential calculus enjoys a non-commutative symbolic calculus which describes taking the adjoints and the composition  as asymptotic sums of symbols~\cite{FR,FF1}. These tools  
 prove efficient for studying  the propagation of oscillations or concentration effects~\cite{FF2} of Schr\"{o}dinger equations, or related questions in PDEs such as the existence of observability inequalities~\cite{FL}.
 
 \smallskip
 
 In the Euclidean case, invariance to leading order by change of variables is an important property of the semiclassical pseudodifferential calculus. This invariance allows for the identification of the symbol as a function on the cotangent space, and is therefore the foundation of the semiclassical calculus on manifolds~\cite{Zwobook}. 
  Here we investigate this property for symbols on graded Lie groups: what happens to the semi-classical pseudodifferential operators introduced in~\cite{FF1} when conjugated by the change of variables induced by  a local smooth diffeomorphism that preserves the filtration of the group.

\smallskip 

It turns out that a smooth function $\Phi$ between two graded Lie groups $G$ and $H$ that preserves the filtration of the groups is uniformly Pansu differentiable (see Theorem~\ref{thm_PDiffFP} below). 
After the publication of the present paper to Journal of Geometric Analysis, it was pointed out to us that this analysis had been performed in greater generality on filtered manifolds in 
\cite[Section 7]{Choi+Ponge}.
However, our analysis here is done from first principles on graded nilpotent Lie groups. Indeed, part of the article consists in revisiting the concept of Pansu differentiability in the context of graded nilpotent Lie groups and its link with being filtration preserving.
  This study allows us to define one-to-one maps between the phase spaces $G\times \widehat G$ and  $H\times \widehat H$ associated with any smooth local diffeomorphisms  $\Phi$ between two nilpotent groups $G$ and $H$ (see Theorem~\ref{thm:inv} below). We investigate the geometric interpretation of these results in the setting of filtered manifolds in the last Section~\ref{sec:app}.

\subsection{Graded groups}\label{sec:intro}

A graded group $G$ is a connected simply connected nilpotent Lie group whose (finite dimensional, real) Lie algebra $\mathfrak g$ admits an $\N$-gradation into linear subspaces, i.e.
$$
\mathfrak g = \oplus_{j=1}^\infty \mathfrak g_j 
\quad\mbox{with} \quad [\mathfrak g_i,\mathfrak g_j]\subseteq \mathfrak g_{i+j}, \;\; 1\leq i\leq j.
$$
With this way of describing $\mathfrak g$, all but a finite number of subspaces $\mathfrak g_j $ are trivial.
We denote by $j=n_G$ the smallest integer such that all the subspaces $\mathfrak g_j $, $j>n_G$,  are trivial.  
If the first stratum $\mathfrak g_1$ generates  the whole Lie algebra, then $\mathfrak g_{j+1}= [\mathfrak g_1,\mathfrak g_j]$ for all $j\in\N_0$  and $n_G$ is the step of the group; the group $G$ is then said to be stratified, and also  (after a choice of basis or inner product for $\mathfrak g_1$)  Carnot.

\smallskip 
 
 The product law on $G$ is derived from the exponential map  $\exp_G : \mathfrak g \to G$ and the Dynkin formula for the Baker-Campbell-Hausdorff formula (see  \cite[Theorem~1.3.2]{FR})
 \begin{align}
 \label{eq_dynkin}
   X*Y &:= \ln_G \left( \exp_G(X)\, \exp_G(Y)\right)
\\& =\sum_{\ell=1}^\infty \frac{(-1)^{\ell-1}}\ell \sum_{
\substack{r,s\in \N_0^\ell\\ r_1+s_1>0,\ldots , r_\ell+s_\ell>0}}
c_{r,s}
\ad^{r_1} X 
\ad^{s_1} Y \ldots  \ad^{r_\ell} X \ad^{s_\ell-1} Y (Y)
\nonumber
 \end{align}
where the coefficients $c_{r,s}$ are known:
$$
c_{r,s} ^{-1} = \Big(\sum_{j=1}^\ell (r_j+s_j)\Big)\,  \Pi_{i=1}^\ell r_i!s_i!
$$
Above, 
the sum over $\ell$ is finite in the nilpotent case. In particular,  the term for which $s_\ell>1$ or $s_\ell=0$ and $r_\ell>1$ is zero, while the term $\ad X\ad^{-1}Y(Y)$ for $s_\ell=0, r_\ell=1$ is understood as $X$. Here $\ln_G$ denotes the inverse map to ${\exp_G}$; we may drop the subscript $G$ for $\exp$ and $\ln$ when the context is clear.

The exponential mapping  is a global diffeomorphism from $\mathfrak g$ onto $G$. Once a basis $X_1,\ldots,X_{n}$
for~$\mathfrak g$ has been chosen, we may  identify 
the points $(x_{1},\ldots,x_n)\in \mathbb R^n$ 
 with the points  $x=\exp(x_{1}X_1+\cdots+x_n X_n)$ in~$G$.
It allows us to define 
the (topological vector) spaces $\mathcal C^\infty(G)$ and $\mathcal S(G)$  of smooth and  Schwartz functions on $G$ identified with $\mathbb R^n$; 
note that the resulting spaces are intrinsically defined as spaces of functions on $G$ and do not depend on a choice of basis. 

The exponential map induces a Haar measure $dx$ on $G$ which is invariant under left and right translations and  defines Lebesgue spaces on~$G$, together with a (non-commutative) convolution for functions 
$f_1,f_2\in\mathcal S(G)$ or in~$L^2(G)$,
$$
 (f_1*f_2)(x):=\int_G f_1(y) f_2(y^{-1}x) dy,\;\; x\in G.
$$

\smallskip

We now construct a basis adapted to the gradation. 
Set $n_j={\rm dim} \, \mathfrak g_j $ for $1\leq j\leq n_G$.
We choose a basis $\{X_1,\ldots, X_{n_1}\}$ of $\mathfrak g_1$ (this basis is possibly reduced to $\emptyset$), then 
$\{X_{n_1+1},\ldots,  X_{n_1+n_2}\}$ a basis of $\mathfrak g_2$
(possibly $\{0\}$)
and so on.
Such a basis $\mathcal B=(X_1, \cdots , X_n)$
of $\mathfrak g$ is said to be adapted to the gradation;
here $n=\dim \mathfrak g = n_1+\ldots +n_G$.  

\smallskip 

The Lie algebra 
 $\mathfrak g$ is  a homogeneous Lie algebra equipped with 
the family of dilations  $\{\delta_r, r>0\}$,  $\delta_r:\mathfrak g\to \mathfrak g$, defined by
$\delta_r X=r^\ell X$ for every $X\in \mathfrak g_\ell$, $\ell\in \N$
\cite{folland+stein_82,FR}.
We re-write the set of integers $\ell\in \N$ such that $\mathfrak g_\ell\not=\{0\}$
into the increasing sequence of positive integers
 $\upsilon_1,\ldots,\upsilon_n$ counted with multiplicity,
 the multiplicity of $\mathfrak g_\ell$ being its dimension.
 In this way, the integers $\upsilon_1,\ldots, \upsilon_n$ become 
 the weights of the dilations and we have $\delta_r X_j =r^{\upsilon_j} X_j$, $j=1,\ldots, n$,
 on the chosen basis of $\mathfrak g$.
 The associated group dilations are defined by
$$
\delta_r(x)=
rx
:=(r^{\upsilon_1} x_{1},r^{\upsilon_2}x_{2},\ldots,r^{\upsilon_n}x_{n}),
\quad x=(x_{1},\ldots,x_n)\in G, \ r>0.
$$
In a canonical way,  this leads to the notions of homogeneity for functions and operators. It also motivates the definition of 
 quasi-norms on $G$ as continuous functions $|\cdot| : G\rightarrow [0,+\infty)$ homogeneous of degree 1
on $G$ which vanishes only at 0. This often replaces the Euclidean norm in the analysis on homogeneous Lie groups.
Any quasi-norm $|\cdot|$ on $G$ satisfies a triangle inequality up to a constant:
$$
\exists C\geq 1, \quad \forall x,y\in G,\quad
|xy|\leq C (|x|+|y|).
$$
Any two homogeneous quasi-norms $|\cdot|_1$ and $|\cdot|_2$ are equivalent in the sense that
$$
\exists C>0, \quad \forall x\in G,\quad
C^{-1} |x|_2\leq |x|_1\leq C |x|_2.
$$
For example, the Haar measure is $Q$-homogeneous
where
$$
Q:=\sum_{\ell\in \mathbb N}\ell n_\ell=\upsilon_1+\ldots+\upsilon_n
$$
 is called the \emph{homogeneous dimension} of $G$.

\subsection{Diffeomorphisms between graded groups}

\subsubsection{Filtration preserving maps}
In this paragraph, we explain the notion of a (smooth or at least~$\mathcal C^1$) map   preserving the natural filtrations between two graded nilpotent groups~$G$ and~$H$.  
We first need to set some notation. For each $x\in G$,
we  denote   
the differential of the left-translation $L_x: G \rightarrow G$  by $x\in G$ at the identity
by 
$$
\tau^G_x := D_0 L_x : \mathfrak g \rightarrow T_x G.
$$
For $y\in H$, we denote by $\tau_y^H$ the analogue map on $H$;
we may omit the superscript $G$ or $H$ if the context is clear. 
These maps allow us to view derivatives between Lie groups as acting on the Lie algebras of the groups: if $\Phi$ is a smooth map from an open set $U$ of $G$ to $H$, 
then we define for each $x\in U$
$$
\mathfrak d_x \Phi:=\left( \tau^H_{\Phi(x)} \right)^{-1}    \circ D_x\Phi \circ \tau^G_x .
$$
By definition, the following diagram commutes :
\begin{equation*}
 \begin{matrix}
   \; &      \;     &    D_x\Phi  &   \;     &\\
   &T_x G & \rightarrow & T_{\Phi(x)} H& \\
\tau^G_x  & \uparrow &               & \uparrow        &  \tau^H_{\Phi(x) } \\
& \mathfrak g & \rightarrow &   \mathfrak h & \\
&                   &    \mathfrak d _x \Phi &       &
\end{matrix}
\ . 
\end{equation*} 
Note that  an equivalent definition for $\mathfrak d_x\Phi : \mathfrak g \to \mathfrak h$ is given by 
\begin{equation}
\label{eq:tgoes0}
\mathfrak d_x\Phi(V)=
\lim_{t\rightarrow 0} t^{-1} 
\ln_H \left(
 \Phi(x)^{-1} \Phi(x{\exp_G}(tV))\right), 
 \qquad
V\in \mathfrak g.
\end{equation}
In this article, we will  use the shorthand:
\begin{equation}
\label{eq_Sshorthand}
 \Phi_x(y)
:= \Phi(x)^{-1} \Phi(xy),
\end{equation}
and for instance \eqref{eq:tgoes0} may be rewritten as 
$$
\mathfrak d_x\Phi(V)=
\lim_{t\rightarrow 0} t^{-1} 
\ln_H \left(
 \Phi_x({\exp_G}(tV))\right)
 =\partial_{t=0}\Phi_x({\exp_G}(tV)),\qquad V\in \mathfrak g.
 $$ 

\begin{definition}
\label{def_filtration_preserving}
Let $\Phi$ be a smooth function from an open set $U$ of $G$ to $H$.
We say that $\Phi$ \emph{preserves the filtration} at $x\in U$  when we have
$$ 
\mathfrak d_x\Phi ( \mathfrak g_j)\subset \mathfrak h_1\oplus \cdots \oplus \mathfrak h_j,
\qquad j=1,2,\ldots
$$
\end{definition}

\subsubsection{Pansu differentiability}

For smooth diffeomorphisms (or at least sufficiently differentiable), the property of preserving the filtration is related to 
 the notion of Pansu differentiability, which we define following~\cite{pansu}:
\begin{definition}
Let $\Phi$ be a map from an open set $U$ of $G$ to $H$, and 
let $x\in U$. 
The function $\Phi$ is \emph{Pansu differentiable at the point} $x$ when for any $z\in G$, 
the following limit exists:
\begin{equation}
\label{eq_lim_defPD}
\lim_{\eps\to 0}  \delta_{\eps^{-1}}\left( \Phi(x)^{-1} \Phi(x\delta_\eps z)\right)= \lim_{\eps\to 0}  \delta_{\eps^{-1}}\left(  \Phi_x(\delta_\eps z)\right). 
\end{equation}
The limit is denoted by $\PD_x \Phi(z)$; the map $z\to \PD_x\Phi(z)$ is called the \emph{Pansu derivative} of $\Phi$ at $x.$
\end{definition}

When defining Pansu differentiability on an open set, we add a further hypothesis of uniformity:
\begin{definition}\label{def:pansu}
Let $\Phi$ be a map from an open set $U$ of $G$ into $H$.
The function $\Phi$ is \emph{uniformly Pansu differentiable on} $U$  when 
\begin{enumerate}
\item $\Phi$ is Pansu differentiable at every point $x\in U$, 
\item the limit in \eqref{eq_lim_defPD} holds locally uniformly on $U\times G.$
\end{enumerate} 
\end{definition}

Part (2) of this definition means that for every point $(x,z)\in U\times G$, there exists a compact neighborhood of $(x,z)$ in $U\times G$ where  the limit in \eqref{eq_lim_defPD} holds uniformly.
Note that this implies that   the limit in \eqref{eq_lim_defPD} holds uniformly on any compact subset of $U\times G$. A similar assumption is made in Section~4 of~\cite{war} with a subtle difference: the limit in \eqref{eq_lim_defPD} is supposed to be locally uniform only in $z\in G$; however, the Pansu derivative at every point is required to be an automorphism of the group $G$. We will see in Section~\ref{sec:FiltPres} that condition (2) in Definition~\ref{def:pansu} implies that the Pansu derivative is a group morphism, and an automorphism when $\Phi$ is a diffeomorphism.

\begin{example}
		Let $\Phi: \mathbb{H}\to \mathbb{R}^3; \;(p, q, t)\mapsto (p, q, t)$ be the identity map from the Heisenberg group to $(\mathbb{R}^3, +)$. Equip both groups with the gradation corresponding to the non-isotropic dilations $\delta_\eps(p, q, t)=(\eps p, \eps q, \eps^2 t).$ One computes, for $x=(p_1, q_1, t_1)$ and  $y=(p_2, q_2, t_2)$,
\begin{align*}
\delta_\eps^{-1}\left(\Phi(x)^{-1}\Phi(x\delta_\eps y)\right)
=\delta_{\eps}^{-1}\left(x\delta_\eps y-x\right)
=(p_2, q_2, t_2+\tfrac{1}{2\eps}(p_1q_2-p_2q_1))
\end{align*}
which converges as $\eps\to 0$ if and only if $p_1q_2=p_2q_1$. In particular, the identity map between $\mathbb{H}$ and $(\mathbb{R}^3, +)$ is Pansu differentiable at the origin but not in a neighborhood of the origin. 
\end{example}

The following property is crucial for our analysis and will be proved in  Section~\ref{sec:FiltPres} below. It states that for smooth functions, the notion of  uniform Pansu differentiability on an open set coincides with preserving filtration above this set:

\begin{theorem}
\label{thm_PDiffFP}
Let $\Phi$ be a smooth function from an open set $U$ of $G$ to $H$.
The map  $\Phi$ is uniformly Pansu differentiable on $U$ if and only if $\Phi$ preserves the filtration at every point $x\in U.$
Besides, for such a map $\Phi$,
the Pansu derivative yields a smooth function  $(x,z)\mapsto \PD_x\Phi(z)$ on $U\times G$. 
\end{theorem}

This result has been proved in~\cite{war} in the case of  Carnot groups. The filtration preserving maps then are contact maps. The Heisenberg group is a prime example of Warhust's result~\cite{war}.
One contribution of this paper is to extend the result to the case of graded  groups with a proof `from first principle'.  
After the publication of the present paper to Journal of Geometric Analysis, it was pointed out to us that this analysis had been performed in greater generality on filtered manifolds in 
\cite[Section 7]{Choi+Ponge}.

\smallskip 

Our proof of Theorem \ref{thm_PDiffFP} will also yields the following properties:

\begin{corollary}\label{cor:iso}
We continue with the notation and setting of Theorem \ref{thm_PDiffFP}.
Let $\Phi$ be a smooth function from an open set $U$ of $G$ to $H$ which is 
uniformly Pansu differentiable on $U$. 
\begin{itemize}
    \item For each $x\in U$, $f\circ \PD_x \Phi\in\mathcal C^\infty(G)$ (resp. $\mathcal S(G)$) if $f\in \mathcal C^\infty(H)$ (resp. $\mathcal S(H)$). The map  
$$
f\longmapsto (x\mapsto f\circ \PD_x\Phi)
$$
yields continuous morphisms of topological vector spaces from $\mathcal C^\infty(H)$ to $\mathcal C^\infty(U,\mathcal C^\infty(G))$
and also from~$\mathcal S(H)$ to $\mathcal C^\infty(U,\mathcal S(G))$.
\item If $\Phi$ is a smooth diffeomorphism from $U$ onto its image, 
then $G$ and~$H$ are isomorphic.
\end{itemize}
\end{corollary}

\subsection{The dual set and the semi-classical pseudodifferential calculus}

\subsubsection{The dual set}
Recall that a  {\it representation} $(\mathcal H_\pi, \, \pi)$ of $G$ is a pair consisting of a Hilbert space~$\mathcal H_\pi$  and a  group morphism~$\pi$ from~$G$ to the set of unitary transforms on $\mathcal H_\pi$.
In this paper, the representations will always be assumed (unitary) strongly continuous, and their Hilbert spaces separable. 
A representation is said to be {\it irreducible} if the only closed subspaces of $\mathcal H_\pi$ that are stable under~$\pi$ are $\{0\}$ and $\mathcal H_\pi$ itself. 
Two representations $\pi_1$ and $\pi_2$ are equivalent  if there exists a unitary transform $\mathbb U$ called an {\it intertwining map} that sends $\mathcal H_{\pi_1}$ on $\mathcal H_{\pi_2}$ with 
$$\pi_1=\mathbb U^{-1}\circ  \pi_2 \circ \mathbb U.$$ 
The {\it dual set} $\widehat G$ is obtained by taking the quotient of the set of irreducible representations by this equivalence relation.  We may still denote by $\pi$ the elements of $\widehat G$ and we keep in mind that different representations of the class are equivalent through intertwining operators. The following properties are straightforwards:

\begin{lemma}[Pullback of Irreps]\label{pullbackOfIrreps}
Let $\theta: G\to H$ be a continuous group homomorphism, and $(\pi, \mathcal{H}_\pi)$ be a representation of $H$. Then
\begin{itemize}
    \item $\pi \circ\theta$ is a representation of $G$,
    \item  $\pi\mapsto \pi\circ \theta$ preserves unitarity and the unitary equivalence class of $\pi$,
    \item $\pi\mapsto \pi\circ\theta$ preserves irreducibility if $\theta$ is surjective.
\end{itemize}
Hence if $\theta$ is surjective, we have a map
\begin{align*}
    \theta^*: \widehat{H}&\longrightarrow\widehat{G}\\
    [\pi]&\longmapsto [\pi\circ\theta].
\end{align*}
\end{lemma}

In particular, any automorphism of $G$ induces an automorphism of $\widehat{G}$, and hence any subgroup of $\text{Aut}(G)$
acts on  $\widehat{G}$.
For instance, the dilations $\delta_r$, $r>0$, on a graded Lie group $G$ provide an action of $\mathbb{R}^+$ 
on $\widehat G$ via
\begin{equation}
\label{eq_def_rpi}
\delta_r \pi (x)
=
\pi(rx),\quad x\in G,\
\pi\in \widehat G, \ r>0.
\end{equation}

\subsubsection{The Fourier transform}
\label{subsubsecFG}
The {\it Fourier transform} of an integrable function $f\in L^1(G)$ at a representation $\pi$ of $G$ is the operator acting on $\mathcal H_\pi$
 via 
 $$
 \mathcal F_G(f)(\pi) =\int_G  f(z)\, (\pi(z))^*\, dz.
 $$
 Note that $\mathcal F_G(f)(\pi)\in \mathcal L(\mathcal H_\pi)$; in this paper, we denote by $\mathcal L(\mathcal H)$ the Banach space of bounded operator on the Hilbert space $\mathcal H$. 
  Note also that if $\pi_1,\pi_2$ are two equivalent representations of $G$ with $\pi_1=\mathbb U^{-1}\circ  \pi_2 \circ \mathbb U$ for some intertwining operator $\mathbb U$, then 
$$\mathcal F_G(f)(\pi_1) =\mathbb U^{-1}\circ \mathcal F_G(f)(\pi_2)\circ \mathbb U .
$$
Hence, this defines the measurable field of operators $\{{\mathcal F}_G(f)(\pi), \pi\in \Ghat \}$ modulo equivalence. 
Here, the unitary dual $\Ghat$ is equipped with its natural Borel structure, 
and the equivalence comes from quotienting the set of irreducible representations of $G$ together with understanding the resulting fields of operators modulo intertwiners.

The {\it Plancherel measure} is the unique positive Borel measure $\mu$ 
on $\widehat G$ such that 
for any $f\in {\mathcal C}_c(G)$, we have the Plancherel formula
\begin{equation}
\label{eq_plancherel_formula}
\int_G |f(x)|^2 dx = \int_{\widehat G} \|\mathcal F_G(f)(\pi)\|_{HS(\mathcal H_\pi)}^2 d\mu(\pi).
\end{equation}
Here $\|\cdot\|_{HS(\mathcal H_\pi)}$ denotes the Hilbert-Schmidt norm on $\mathcal H_\pi$.
This implies that the group Fourier transform extends unitarily from 
$L^1(G)\cap L^2(G)$ to $L^2(G)$ onto 
$L^2(\widehat G):=\int_{\widehat G} \mathcal H_\pi \otimes\mathcal H_\pi^* d\mu(\pi)$
which we identify with the space of $\mu$-square integrable fields on $\widehat G$.

\smallskip 

The Plancherel formula also  yields  an {\it inversion formula} for any~$ f \in {\mathcal S}(G)$ and~$x\in G$: 
\begin{equation}
\label{inversionformula} f(x)
=  \int_{\widehat G} {\rm{Tr}}_{\mathcal H_\pi} \, \Big(\pi(x) {\mathcal F}_Gf(\pi)  \Big)\, d\mu(\pi) \,,
\end{equation}
where ${\rm Tr}_{\mathcal H_\pi}$ denotes the trace of operators in ${\mathcal L}(\mathcal H_\pi)$, the set of bounded linear operators on $\mathcal{H}_\pi$.
This   formula makes sense since, for $f \in {\mathcal S}(G)$, the operators ${\mathcal F}f(\pi)$ are trace-class and $\int_{\widehat G} {\rm{Tr}}_{\mathcal H_\pi} \, \Big| {\mathcal F}f(\pi)  \Big|d\mu(\pi)$ is finite.

\subsubsection{Semiclassical analysis} \label{subsubsec:sc}

For an open set $U\subset G$, we denote by $\mathcal A_0(U\times \widehat{G})$ the space of {\it symbols}
$\sigma = \{\sigma(x,\pi) : (x,\pi)\in U\times \widehat G\}$ of the form 
$$
\sigma(x,\pi)=\mathcal F_G \kappa_x (\pi) = \int_G \kappa_x(y) (\pi(x))^* dy, 
$$
where $\kappa\in \mathcal C_c^\infty(U,\mathcal S( G))$; that is,  the map $x\mapsto \kappa_x$ is  compactly supported on~$U$ and valued in the set of Schwartz class functions $S(G)$ with Schwartz norms depending smoothly on $x\in U$. We call $x\mapsto \kappa_x$ the \emph{convolution kernel} of $\sigma$. 

As the Fourier transform is injective, it yields a one-to-one correspondence between $\mathcal A_0(U\times\widehat{G})$ and  $\mathcal C_c^\infty(U,\mathcal S( G))$. In this way,   $\mathcal A_0(U\times\widehat{G})$ inherits from  $\mathcal C_c^\infty(U,\mathcal S( G))$ the structures of topological vector space and smooth compactly supported section of the Schwartz bundle over $U$.
These structures are easier to grasp than being  a set of measurable fields of operators on $U\times \Ghat$ modulo intertwiners as  described in Section \ref{subsubsecFG}.
The results of this paper develop a deeper geometric interpretation of the space of symbols, here  $\mathcal A_0(U\times \Ghat)$, in terms of densities; this will be discussed in Section~\ref{sec:app}.

\smallskip

With the symbol $\sigma\in\mathcal A_0(U\times \Ghat)$, we associate the (family of) {\it semi-classical pseudodifferential operators} ${\rm Op}_\eps(\sigma)$, $ \eps\in (0,1]$, defined via 
$$
{\rm Op}_\eps(\sigma) f(x) = \int_{\pi\in\widehat G} {\rm Tr}_{{\mathcal H}_\pi} \left(  \pi(x) \sigma(x,\delta_\eps\pi) {\mathcal F}_G f(\pi)  \right)d\mu(\pi),\;\;
f\in \mathcal S(G), \, x\in U, 
$$
or equivalently, in terms of the convolution kernel $\kappa_x =\mathcal F^{-1}_G \sigma(x,\cdot)$,
$$
{\rm Op}_\eps(\sigma) f(x) = f* \kappa^{(\eps)}_x (x),\quad 
f\in \mathcal S(G), \, x\in U,
$$
where $\kappa^{(\eps)}_x$ is the following rescaling of the convolution kernel:
$$
\kappa^{(\eps)}_x (y) := \eps^{-Q} \kappa_x(\delta_\eps ^{-1} y).
$$
This second expression for ${\rm Op}_\eps(\sigma)$ explains the choice of vocabulary for the convolution kernel of a symbol. 

The rescaled convolution kernel is different from the integral kernel of ${\rm Op}_\eps(\sigma)$, which is given by:
$$
(x,y)\mapsto \eps^{-Q} \kappa_x (\delta_\eps^{-1} (y^{-1} x)).
$$
We emphasize the convolution kernel over the integral kernel, although the latter is  used in the pioneer work~\cite{VanErp} on pseudodifferential theory on filtered manifolds (see also~\cite{VeY1,VeY2}).
The reason is that our approach here is different and aims to be symbolic. It is inspired from semi-classical analysis in Euclidean setting and based on the fact that the symbols obtained via the Fourier transforms of convolution kernels  proved a flexible tool for  dealing with applications (see~\cite{FL,FF2}).

\smallskip 

Our semi-classical pseudodifferential theory may be used to analyse the oscillations of  families~$(u^\eps)_{\eps>0}$  that are bounded in $L^2(G)$. One considers the functionals $\ell_\eps$ defined on $\mathcal A_0(U\times \widehat G)$ by 
$$\ell_\eps(\sigma)= \left({\rm Op}_\eps (\sigma) u^\eps,u^\eps\right)_{L^2(G)} ,\;\;\sigma\in \mathcal A_0(U\times \widehat G).$$
The limit points of $\ell_\eps$
 as $\eps$ goes to $0$ have some structures. When the family $(u^\eps)_{\eps>0} $ is $L^2$-normalized and after possibly further extraction of subsequences, the limit points of $\ell_\eps$ define states of the $C^*$-algebra $\mathcal A(U\times \widehat G)$ obtained by completion of $\mathcal A_0(U\times \widehat G)$ for the norm
 $$\| \sigma\|_{L^\infty (U\times \Ghat)} :=\sup_{(x,\pi)\in U\times \widehat G} \| \sigma(x,\pi)\|_{\mathcal L(\mathcal H_\pi)}.$$
For describing the structure of these limit points, 
we consider the set of pairs $(\gamma,\Gamma)$ where $\gamma$ is a positive Radon measure on~$U\times \widehat G$ 
	and 
	$$\Gamma=\{\Gamma(x,\pi)\in {\mathcal L}({\mathcal H}_\pi): \; (x,\pi)\in U\times \widehat G\}$$
	 is a positive $\gamma$-measurable field of trace-class operators satisfying 
	$$\int_{U\times \widehat G}  {\rm Tr}_{\mathcal H_\pi}\left( \Gamma(x,\pi)\right) d\gamma(x,\pi)<+\infty. $$
This set is then equipped with the  equivalence relation: 
$(\gamma,\Gamma) \sim( \gamma',\Gamma')$  if there  exists a measurable function $f:U\times \widehat G\to \mathbb C\setminus\{0\}$ such that 
$$
d\gamma' =f  d\gamma\;\;{\rm  and} \;\;\Gamma'= {f}^{-1} \Gamma
$$ 
for $\gamma$-almost every $(x,\pi)\in U\times \widehat G$.
The equivalence class of $(\gamma,\Gamma)$ is denoted by 
$\Gamma d\gamma$, it is called a positive vector-valued measure. 
We denote  by $\mathcal M_{ov}^+(U\times \widehat G) $ the set  of these equivalence classes.
 The positive continuous linear functionals of the $C^*$-algebra
$\mathcal A(U\times \widehat G)$ is naturally identified with $\mathcal M_{ov}^+(U\times \widehat G) $.

\smallskip 

If  $(u^\eps)_{\eps>0}$ is a bounded family of $L^2(U)$, 
there exist a sequence $(\eps_k)_{k\in \mathbb N}$ in $(0,+\infty)$ with  $\eps_k\Tend{k}{+\infty}0$
and a  pair $\Gamma d\gamma \in{\mathcal M}_{ov}^+(U\times \widehat G)$
 such that we have
$$
\forall \sigma\in \mathcal A_0(U\times \widehat G),\;\; 
\left({\rm Op}_{\eps_k} (\sigma) u^{\eps_k},u^{\eps_k}\right)_{L^2(G)}\Tend {k}{+\infty} \int_{G\times \widehat G} {\rm Tr}_{\mathcal H_\pi}\left(\sigma(x,\pi) \Gamma(x,\pi)\right)d\gamma(x,\pi).$$
Given the sequence $(\eps_k)_{k\in \mathbb N}$,  one has 
$$\int_{U\times \widehat G} {\rm Tr}_{\mathcal H_\pi}\left(\Gamma(x,\pi)\right)d\gamma(x,\pi) \leq \limsup_{\eps>0} \|u^\eps\|_{L^2(U)}^2.$$
The positive vector-valued measure $\Gamma d\gamma$ is called a {\it semi-classical measure} of the family $(u^\eps)$ for the sequence $\eps_k$.

\smallskip 

Our aim is to analyze how these notions can be transferred from (an open subset of) a graded group $G$ to another one $H$ via a smooth local diffeomorphism that preserves the filtration of the group.

\subsection{Main result}

Let us consider two graded groups, $G$ and $H$, and    a smooth diffeomorphism, $\Phi$, from an open set $U$ of $G$ to $H$ that is filtration preserving and thus uniformly Pansu differentiable by Theorem~\ref{thm_PDiffFP}. 

We denote by $J_\Phi$ the Jacobian of $\Phi$ with respect to Haar measures chosen on $G$ and $H$, and we consider the operator $\PullTwo$  which associates to a function~$f$ defined on $\Phi(U)\subset{H}$ the function 
\begin{equation}\label{eq:PullTwo}
\PullTwo(f) :=  J_\Phi^{1/2} f\circ \Phi
\end{equation}
 defined on  $U\subset G$.  
The map $\PullTwo$ is a unitary operator from $L^2(\Phi(U))$ onto $L^2(U)$.

\smallskip 

Conjugation by $\PullTwo$ pulls back any $S\in \mathcal L (L^2(\Phi(U)))$ to an operator 
$$
\PullTwo \circ S \circ \PullTwo^{-1} \in \mathcal L(L^2(U)).
$$
We want to study the case of $S={\rm Op}_\eps (\sigma)$, $\sigma \in\mathcal A_0(\Phi(U)\times\Hhat)$. 

\smallskip 

The map  $\Phi$ also induces a transformation $\PullOne$ between convolution kernels. 
Indeed, for a function $\kappa:x\mapsto \kappa_x$ in $\mathcal C_c^\infty(\Phi(U), \mathcal S(H))$,
we associate the function $\PullOne \kappa$ defined via
\begin{equation}\label{def:map_kernel}
(\PullOne \kappa)_x(z) = J_\Phi(x)  \kappa_{\Phi(x)} ( \PD_x\Phi(z)),\;\;\forall (x,z)\in U\times G.
\end{equation}
By Corollary \ref{cor:iso}, $\PullOne \kappa$ is  in $\mathcal C_c^\infty(U, \mathcal S(G))$ and $\PullOne$ is an isomorphism of topological vector spaces from  
$\mathcal C_c^\infty(\Phi(U), \mathcal S(H))$ onto $\mathcal C_c^\infty(U, \mathcal S(G))$. 
The geometric meaning of this map is discussed in Section~\ref{sec:app}.

\smallskip 

By Lemma \ref{pullbackOfIrreps}, the map 
\begin{equation}\label{unitaryDualMap}
    \widehat{\mathbb{G}}\Phi: 
    \left\{
    \begin{array}{rcl}
    U\times\widehat{G}&\longrightarrow& \Phi(U)\times\widehat{H}\\
    (x, \pi)&\longmapsto& \left(\Phi(x), \pi\circ (\PD_x\Phi)^{-1}\right)
\end{array}\right.\quad, 
\end{equation}
is well-defined. 
Moreover, it induces  an isomorphism of topological vector spaces
$$
(\widehat{\mathbb{G}}\Phi)^*: \mathcal{A}_0(\Phi(U)\times\Hhat)\rightarrow\mathcal{A}_0(U\times\Ghat),
$$ 
defined in the following way:
for $\sigma\in \mathcal A_0(\Phi(U)\times\Hhat)$, 
the symbol $(\widehat{\mathbb{G}}\Phi)^*\sigma$ given by 
\begin{align}\label{unitaryPullBack}
 (\widehat{\mathbb{G}}\Phi)^*\sigma(x, \pi):=\sigma\left(\Phi(x), \pi\circ\PD_x\Phi^{-1}\right),  \quad \, (x,\pi)\in U\times\Ghat,
\end{align}
 is in  $\mathcal A_0(U\times \Ghat)$. Indeed, its convolution kernel is 
$\PullOne\kappa$ when  $\kappa:x\mapsto \kappa_x$ denotes the convolution kernel of~$\sigma$.

\begin{theorem}\label{thm:inv}
Let $G$ and $H$ be two graded groups, and  $\Phi$  a smooth diffeomorphism from an open set $U$ of $G$ to $H$.
Assume that $\Phi$ is filtration preserving on~$U$ (and thus uniformly Pansu differentiable on~$U$). Let~$\sigma\in\mathcal A_0(\Phi(U)\times\Hhat)$, then 
in~$\mathcal L(L^2(U))$,
$$
\PullTwo\circ  {\rm Op}_\eps (\sigma)\circ \PullTwo^{-1} ={\rm Op}_\eps\left((\widehat{\mathbb{G}}\Phi)^* \sigma \right) +O(\eps).
$$  
\end{theorem} 

Note that when we assume that the smooth diffeomorphism $\Phi$ is uniformly Pansu differentiable on $U$, this implies that the graded groups $G$ and $H$ are isomorphic (see Corollary~\ref{cor:iso}).

\smallskip 

Theorem~\ref{thm:inv} is proved in Section~\ref{sec:inv} below. It crucially relies on  the Pansu differentiability of filtration preserving maps (see  Section \ref{sec:FiltPres}). It also has straightforward consequences on semi-classical measures that we now present.  We associate with $\Phi$ a map from $\mathcal M^+_{ov}(U\times \widehat G)$ into $\mathcal M^+_{ov}(\Phi(U)\times \widehat H)$ which maps $\Gamma d\gamma\in \mathcal M^+_{ov}(U\times \widehat G)$ on 
$\Gamma^\Phi d\gamma^\Phi\in \mathcal M^+_{ov}(\Phi(U)\times \widehat H)$ defined by 
\begin{align*}
\int_{ \Phi(U)\times\widehat H} {\rm Tr} _{\mathcal H_{\pi'}} \left(\sigma (x',\pi')  \Gamma^\Phi(x',\pi')\right) d\gamma^\Phi(x',\pi') = 
\int_{ U\times\widehat G}   {\rm Tr} _{\mathcal H_\pi} \left( ((\widehat {\mathbb{G}}\Phi)^* \sigma ) (x,\pi) \Gamma(x,\pi)\right) d\gamma(x,\pi).
\end{align*}
Then, the result of Theorem~\ref{thm:inv} implies the next corollary.

\begin{corollary}
Let $G$ and $H$ be two graded groups, and let $\Phi$  a smooth diffeomorphism from an open set $U$ of $G$ to $H$.
Assume that $\Phi$ is filtration preserving on~$U$ and thus uniformly Pansu differentiable on~$U$. Let $(f^\eps)_{\eps>0}$ be a bounded family in $L^2(U)$ and $\Gamma d\gamma$ be a semi-classical measure of $(f^\eps)_{\eps>0}$ for the sequence $\eps_k$. Then, $\Gamma^\Phi d\gamma^\Phi $ is a semi-classical measure of the family $\left(\PullTwo^* f^\eps\right)_{\eps>0}$ for the sequence~$\eps_k$. 
\end{corollary}

\smallskip

The proof of Theorem~\ref{thm:inv} is developed in Section~\ref{sec:inv} and heavily uses the results of Section~\ref{sec:FiltPres} via  Theorem~\ref{thm_PDiffFP} and its consequences in Corollary \ref{cor:iso}. 
Similar results should hold in the
 microlocal case where no specific scale $\eps$ is chosen, as developed in~\cite{FF0}. 
Finally, we develop the geometric interpretation of the convolution kernel in Section~\ref{sec:app}.

\medskip 

\noindent{\bf Acknowledgments.} The authors thank Antoine Julia and Pierre Pansu for inspiring discussions.  


\section{Filtration preserving diffeomorphisms}\label{sec:FiltPres}

 Our aim in this section is to prove the equivalence between the uniform Pansu differentiability on an open set  and the preservation of  the  filtration above this set, see Theorem \ref{thm_PDiffFP}. We will start with introducing some notations and concepts in order to give a precise meaning to these properties.

\subsection{Pansu differentiability}\label{sec:PD}

In this section, 
we recall basic properties of the Pansu derivative:

\begin{lemma}
\label{rem:bound} 
Let $\Phi$ be a map from an open set $U$ of a graded Lie group $G$ to another graded Lie group $H.$
\begin{enumerate}
\item Let $x\in U$. If $\Phi$ is Pansu differentiable at the point $x$, then the Pansu derivative at $x$ is  a $1$-homogneous map $z\mapsto \PD_x\Phi(z)$ from $G$ to $H$ satisfying $\PD_x\Phi(0_G)=0_H$. 
Moreover, 
the function $\Phi$ is continuous at $x$. 
\item If $\Phi$ is uniformly Pansu differentiable on $U$, then the Pansu derivative at every point  $x\in U$  is a group morphism from $G$ to $H:$
$$
\forall z_1,z_2\in G,\quad \PD_x\Phi(z_1) \circ \PD_x\Phi(z_2) = \PD_x\Phi(z_1z_2).
$$
\end{enumerate}
\end{lemma}

\begin{proof}
(1) This comes from the definition and the observation that  for $r>0$, 
$$ \delta_{\eps^{-1} }\left(\Phi(x)^{-1} \Phi(x\delta_\eps \delta_rz)\right)=
 \delta_r \left[ \delta_{(\eps r)^{-1} }\left(\Phi(x)^{-1} \Phi(x\delta_{\eps r} z)\right)\right].$$
 Passing to the limit as $\eps$ goes to $0$, we obtain 
 $$
\PD_x \Phi(\delta_r z) = \delta_r \left[ \PD_x \Phi(z)\right].
 $$
 Then, fixing a homogeneous quasi-norm on $G$ and using that $z=\delta_{|z| } \underline z$ with $|\underline z|=1$, one writes 
 $$\Phi(x)^{-1} \Phi(xz)= \delta_{|z|} \left(  \delta_{|z|^{-1} }\Phi(x)^{-1} \Phi(x\delta_{|z|}\underline z)\right) =  \delta_{|z|} \left(  \PD_x(\underline z)+O(|z|)\right)$$
 Therefore, $\Phi(x)^{-1} \Phi(xz)\Tend{|z|}{ 0} 0$, whence the continuity of $x\mapsto \Phi(x)$. 
 
 \medskip 

\noindent (2) We write 
$$\delta_{\eps^{-1}} \left( \Phi(x)^{-1} \Phi(x\delta_\eps (z_1z_2))\right) = 
 \delta_{\eps^{-1}} \left( \Phi(x)^{-1} \Phi(x\delta_\eps (z_1))\right)
 \delta_{\eps^{-1}} \left( \Phi(x\delta_\eps z_1)^{-1} \Phi(x\delta_\eps (z_1)\delta_\eps (z_2))\right)$$
 and use the uniformity of the convergence. 
\end{proof}

The hypothesis of uniformity for the Pansu differentiability is needed 
to show that the Pansu derivative is a group morphism (see Part (2) above) but also for the following composition property:

\begin{lemma}\label{lem:comp}
Let $F,G,H$ be three graded Lie groups, let $\Phi$ be a map from an open set $U$ of $G$ to $H$, and let $\Psi$ be a map from an open set $U'$ of $F$ to $G$. 
We assume that $\Psi(U')\subset U$.   
If $\Phi$ and $\Psi$ are uniformly Pansu differentiable on $U$ and $U'$ respectively, 
then their composition $\Phi\circ \Psi$ is uniformly Pansu differentiable on $U'$ with 
$$
\forall w\in U', \ z\in F\qquad
\PD_w (\Phi\circ \Psi )(z) = \PD_{\Psi(w)} \Phi ( \PD_w \Psi(z)).
$$ 
\end{lemma}

\begin{proof}
Write
\begin{align*}
\tilde z_\eps:= 
\delta_{\eps^{-1}}\left(\Psi(w)^{-1} \Psi(w\delta_\eps z)\right),
\end{align*}
for $(w, z)$ ranging in a compact subset of $U'\times F$ and $\eps\in (0, 1]$. Observe that
\begin{align}
\delta_{\eps ^{-1}} \left(\Phi\circ \Psi(w)^{-1} \Phi\circ \Psi(w\delta_\eps z)\right) = 
\delta_{\eps ^{-1} }\left( \Phi(\Psi(w))^{-1} \Phi\left(\Psi(w) \delta_\eps \tilde z_\eps \right)\right).\label{RHS}
\end{align}

By continuity of $\Phi$, and uniform Pansu differentiability of $\Psi$, the pair $(\Psi(w), \tilde{z}_\eps)$  ranges in a compact subset of $\Psi(U')\times G$ on which the map
\begin{align*}
(w', z')\mapsto \delta_\eps^{-1}\left(\Phi(w')^{-1}\Phi(w'\delta_\eps z')\right)
\end{align*}
converges uniformly as $\eps\to 0$ (by uniform Pansu differentiability of $\Phi$). Since
$\lim_{\eps\to 0} \tilde z_\eps =\PD_w\Psi(z)$,
the limit of \eqref{RHS}  as $\eps\to 0$ is $\PD_{\Psi(w)} \Phi ( \PD_w \Psi(z))$.
\end{proof}

\begin{remark}
\label{rem_rem:bound}
Let us consider a map $\Phi$ from an open set $U$ of a graded group $G$ to a graded group~$H$
such that  $\Phi$ is a bijection from $U$ onto its image $\Phi(U)$ which is open. 
If $\Phi$ and its inverse $\Phi^{-1}$ are uniformly Pansu differentiable on $U$ and $\Phi(U)$ respectively, 
then 
we may apply Lemmata \ref{rem:bound} and~\ref{lem:comp}  and use $\PD_x{\rm Id}={\rm Id}$ to obtain
$$
\forall x\in U\quad (\PD_x \Phi )^{-1}= \PD_{\Phi(x) }( \Phi^{-1}).
$$
In particular, the groups $G$ and $H$ are isomorphic. 
\end{remark}


\subsection{Filtration preserving smooth maps}

In this section, we study how to characterize smooth maps that are filtration preserving.

A matrix-valued viewpoint will be helpful for a deeper understanding of  Definition~\ref{def_filtration_preserving}. 
For a smooth function $\Phi$ from an open set $U$ of $G$ to $H$,
we denote  by  $M_\Phi(x)$ the matrix of $\mathfrak d_x \Phi$ for the bases 
$$\mathcal B=(X_1,X_2, \cdots , X_{\dim G} )\;\;\mbox{and}\;\;\mathcal C=(Y_1,Y_2, \cdots , Y_{\dim H})$$
 of $\mathfrak g= \oplus_{j=1}^\infty \mathfrak g_j $ and $\mathfrak h = \oplus_{j=1}^\infty \mathfrak h_j $  adapted to the respective gradations (see Section \ref{sec:intro}).
 This matrix can be written by blocks $M_{\Phi, i,j}(x)$   associated with the gradation. 
 As the map $\mathfrak d_x \Phi$ is linear, 
 in order to identify the blocks $M_{\Phi,i,j}(x)$, it is enough to let~$V$ vary in  $\mathfrak g_j$ and calculate  the projection on $\mathfrak h_j$ of $\mathfrak d_x\Phi(V)$.  
For this, we denote by $\pr_{\mathfrak h,j}$
the projection onto $\mathfrak h_j$ along $\oplus_{j'\not = j} \mathfrak h_{j'}$.
We may allow ourselves to remove the subscript~$\mathfrak h$ (and write $\pr_{\mathfrak j}$ instead of $\pr_{\mathfrak h,j}$) when the context is clear.
 
 Let us illustrate this point with the following equivalences:
\begin{lemma}
\label{lem_eq_filtration_preserving}
Let $\Phi$ be a smooth function from an open set $U$ of $G$ to $H$, and let $x\in U$. 
\item The following are equivalent:
\begin{itemize}
\item[(i)] $\Phi$ preserves the filtration at $x$,
\item[(ii)] the matrix $M_\Phi(x)$ defined above is block-upper-diagonal in the sense that all the blocks $M_{\Phi,i,j}(x)$, $i>j$, strictly below the diagonal are 0,
\item[(iii)] we have for any $i>j$
$$
\forall \ V\in\mathfrak g_j,\qquad 
\pr_{\mathfrak h,i}( \mathfrak d_x\Phi(V))=0.
$$
\end{itemize}
\end{lemma}
 
We obtain easily the following implication between preserving the filtration and uniform Pansu differentiability: 
\begin{lemma}
\label{lem_PD_filtration}
Let $\Phi$ be a smooth function from an open set $U$ of $G$ to $H$, and let $x\in U$. 
If $\Phi$ is  Pansu differentiable at $x$, then $\Phi$ preserves the filtration at  $x.$
\end{lemma}

\begin{proof} Let  $V\in \mathfrak g_j$.
We have by taking $t=\eps^j$ in~\eqref{eq:tgoes0},
  $$
\pr_{\mathfrak h,i} (\mathfrak d_x\Phi(V))
  =
  \lim_{\eps \rightarrow 0}  
  \pr_{\mathfrak h,i}\left(\eps^{-j}  \ln_H(\Phi_x({\exp_G}(\eps^j V)))\right)
  $$
while the properties of dilations yield
  $$ 
  \pr_{\mathfrak h,i}\left(\eps^{-j}  \ln_H(\Phi_x({\exp_G}(\eps^j V)))\right)
  =
   \eps^{i-j}\pr_{\mathfrak h,i} \circ \ln_H \left( \delta_{\eps^{-1}} \Phi_x({\exp_G}(\delta_\eps V))\right).
  $$
As $\Phi$ is Pansu differentiable at $x$, the argument inside $\pr_{\mathfrak h,i}\circ \ln_H$ above has a limit. 
 Hence, if $i>j$, we have $  \pr_{\mathfrak h,i} (\mathfrak d_x\Phi(V))=0$ and we conclude with Lemma \ref{lem_eq_filtration_preserving}.
\end{proof}

In the next sections, we will analyze the reverse implication to the one in Lemma \ref{lem_PD_filtration}, using this matrix-valued point of view; this will give Theorem~\ref{thm_PDiffFP}. 


\subsection{Characterization of Pansu differentiability for smooth maps}
If $\Phi$ is Pansu differentiable at $x\in U$, 
we set 
$$
\pd_x \Phi  := {\ln_H} \circ \PD_x \Phi \circ {\exp_G},
$$
so that we have the following diagram :
\begin{equation*}
 \begin{matrix}
   \; &      \;     &    \PD_x\Phi  &   \;     &\\
   &G & \rightarrow &  H& \\
{\exp_G}  & \uparrow &               & \uparrow        &  {\exp_H} \\
& \mathfrak g & \rightarrow &   \mathfrak h & \\
&                   &    \pd _x \Phi &       &
\end{matrix}
\ . 
\end{equation*} 

This defines the map $\pd_x \Phi:\mathfrak g \to \mathfrak h$.
An equivalent definition for this map is given by 
\begin{align}
\label{eq:tgoes0PD} \pd_x\Phi(V)
&=
\lim_{t\rightarrow 0} \delta_{t^{-1}} 
\ln_H \left(
 \Phi(x)^{-1} \Phi(x{\exp_G}(\delta_t V))\right)
=
\lim_{t\rightarrow 0} \delta_{t^{-1}} 
\ln_H \left(
 \Phi_x({\exp_G}(\delta_t V))\right), 
\end{align}
for $V\in \mathfrak g$, having used the shorthand \eqref{eq_Sshorthand}.
Clearly, $\Phi$ is Pansu differentiable at $x\in U$ if and only if the limit in \eqref{eq:tgoes0PD} exists for all $V\in \mathfrak g$, and it is uniformly Pansu differentiable on $U$ if and only if these limits hold locally uniformly on $U\times \mathfrak g.$ 

\medskip

The map $\pd_x\Phi$ may not be linear in general but it will be under mild hypotheses. 
Indeed, when $z\mapsto \PD_x\Phi(z)$ is a continuous group morphism, 
for instance when $\Phi$ is uniformly Pansu differentiable on an open neighbourhood of $x$,
then $\pd_x\Phi$ is linear with 
$$
\PD_x\circ {\exp_G} (V) = ({\exp_H} \circ\, \pd_x \Phi) (V)
\qquad\mbox{and}\qquad
\pd_x\Phi(V) = \partial_{t=0} \PD_x\circ {\exp_G} (tV), \qquad V\in \mathfrak g.
$$

\medskip

As above, we can adopt a matrix-valued point of view and define 
$\PM_\Phi(x)$ to be the matrix whose columns are the images of $\mathcal B$ by $\pd_x\Phi$, that is, the vectors $\pd_x\Phi (X_j)$, $j=1,\ldots, \dim G$, expressed in the basis $\mathcal C$.
The (rectangular)  matrix $\PM_\Phi(x)$ is of the same size as $M_\Phi(x)$.
It makes sense to look at $\PM_\Phi(x)$ and its blocks $\PM_{\Phi,i,j}(x)$ associated with the gradation, or rather to the quantities $\pr_{\mathfrak h,i} (\pd_x\Phi(V)),$  $V\in \mathfrak g_j$. 
As a consequence of the definitions of the objects involved and of homogeneous properties, for each $V\in \mathfrak g_j$ with $i,j=1,2,\ldots,$
\begin{equation}\label{eq:coefPDPhi}
\lim_{\eps\to 0 }
\eps^{-i}\pr_{\mathfrak h,i} \circ \ln_H (\Phi_x({\exp_G} (\eps^j V ))
=
\pr_{\mathfrak h, i} (\pd_x\Phi(V)).
\end{equation}
Actually, very few of the quantities $\pr_{\mathfrak h,i} (\pd_x\Phi(V)),$  $V\in \mathfrak g_j$, are non-zero:

\begin{lemma}
\label{lem_MPD}
Let $\Phi$ be a smooth function from an open set $U$ of $G$ to $H$, and let $x\in U$.
We always have (regardless of whether $\Phi$ is uniformly Pansu differentiable or preserves the filtration)
for $V\in \mathfrak g_j $
$$
\lim_{\eps\to 0 }
\eps^{-i}\pr_{\mathfrak h,i} \circ \ln_H (\Phi_x({\exp_G} (\eps^j V ))
=
\left\{\begin{array}{ll} 
0 & \mbox{when}\ j>i,\\
\pr_{\mathfrak h,i}(\mathfrak d_x\Phi(V))& \mbox{when}\ j=i.
\end{array}\right. 
$$
\end{lemma}

\begin{proof}[Proof of Lemma~\ref{lem_MPD}]
For each $x\in U$ and $V\in \mathfrak g_j$,    we have
$$
\eps^{-i}\pr_{\mathfrak h,i} \circ \ln_H (\Phi_x({\exp}_G (\eps^j V ))
=
  \eps ^{j-i}\pr_{\mathfrak h,i} \left(\eps^{-j} \ln_H (\Phi_x({\exp}_G(\eps^j V)))\right),
$$
and the last argument of $\pr_{\mathfrak h,i}$ tends to $\mathfrak d_x\Phi(V)$ as $\eps$ goes to $0$.
\end{proof}

\begin{remark}\label{rem_MPD}
\begin{enumerate}
\item By Lemma \ref{lem_MPD} and equation~\eqref{eq:coefPDPhi}, if $\Phi$ is Pansu differentiable at $x$ then 
for each $V\in \mathfrak g_j$ with $i,j=1,2,\ldots,$
$$
\pr_{\mathfrak h,i} (\pd_x\Phi(V))
=\lim_{\eps\to 0 }
\eps^{-i}\pr_{\mathfrak h,i} \circ \ln_H (\Phi_x({\exp}_G (\eps^j V ))
=
\left\{\begin{array}{ll} 
0 & \mbox{when}\ j>i,\\
\pr_{\mathfrak h,i}(\mathfrak d_x\Phi(V))& \mbox{when}\ j=i.
\end{array}\right. 
$$
\item Let us give a matrix interpretation of Part (1). 
The matrix $\PM_\Phi(x)$ defined above is block-lower-diagonal in the sense that all the blocks $\PM_{\Phi,i,j}(x)$, $i<j$, strictly above the diagonal are 0. Furthermore, the diagonal blocks coincide with those of $M_\Phi(x)$. 
We will see later that one can say more (see Theorem \ref{thm:LargeStatement}). 
 This matrix interpretation does not depend on the bases $\mathcal B$ and $\mathcal C$ chosen to describe the matrix as long as they are adapted to the gradations. 
\end{enumerate}
\end{remark}

\medskip 

The existence of the limits on the left-hand side of \eqref{eq:coefPDPhi} turns out to also be a sufficient condition for a smooth map to be Pansu differentiable, and thus gives a 
characterization of smooth Pansu differentiable maps (or at least for $\mathcal C^k$-maps for some $k$ large enough).

\begin{proposition}
\label{prop_MPD}
Let $\Phi$ be a smooth function from an open set $U$ of $G$ to $H$.
The function $\Phi$ is uniformly Pansu diffferentiable on $U$ 
if and only if 
for each $x\in U$, $V\in \mathfrak g_j$, with $i,j=1,2,\ldots,$
the limit  of 
\begin{equation}
\label{eq_lim_of_epslambdaiVj}
\eps^{-i}\pr_{\mathfrak h,i} \circ \ln_H (\Phi_x({\exp_G} (\eps^j V )),
\end{equation}
exists as $\eps$ goes to 0   and, 
 all these limits hold locally uniformly on $U\times \mathfrak g_j$ for each $i=1,2,\ldots$
\end{proposition}

\begin{proof}[Proof of Proposition \ref{prop_MPD}]
We drop the indices of the groups and Lie algebras in the notation for the logarithmic and exponential maps  and the projections. 

In view of Remark~\ref{rem_MPD} (1), 
 it  remains to show the reverse implication. 
 Hence, we assume that the  limits in~\eqref{eq_lim_of_epslambdaiVj}
   exist and we aim at proving that 
for any 
$x\in U$, 
$V\in \mathfrak g$, 
 $i=1,2,\ldots$, 
the limit of 
$\pr_i \circ \ln_H 
 \delta_{\eps^{-1}} \left(
 \Phi_x({\exp}_G(\delta_\eps V))\right)$
 exists and that this holds locally uniformly with respect to $(x,V)$ in $U\times \mathfrak g$. 
We will prove this recursively on $i=1,2,\ldots$ 
First, we need to set some  conventions. 

\smallskip

Let us recall that the map $\Theta:\R^{\dim G} \to \R^{\dim G}$ given by the following exponential coordinates of the second kind on $\mathfrak g$
$$
\Theta (V) = (V_1,\ldots, V_{n_G}) \in \mathfrak g_1 \oplus \ldots \oplus \mathfrak g_{n_G}\quad\mbox{where}\quad
{\exp} (V) = {\exp} (V_1)  \ldots   {\exp} (V_{n_G}), 
$$
is a global diffeomorphism of $\R^{\dim G}$, and that 
$$
\Theta (\delta_\eps V) = (\eps V_1, \ldots, \eps^{n_G} V_{n_G}).
$$
We will use the equality
\begin{align*}
\Phi_x({\exp}V) 
& = 
\Phi_x({\exp}V_1)  \ 
\Phi_{x\, {\exp}V_1}({\exp}V_2) \ 
\Phi_{x\, {\exp}V_1\, {\exp}V_2}({\exp}V_3) 
\ldots
 \Phi_{x\, {\exp}V_1\ldots  {\exp}V_{n_G-1}}({\exp}V_{n_G}) ,
\end{align*}
which yields
\begin{align}\label{eq_PhixV_Vj}
\Phi_x({\exp}\delta_\eps V) 
 = 
\Phi_x({\exp} (\eps V_1))  \ 
&\Phi_{x\, {\exp}(\eps V_1)}({\exp}(\eps^2 V_2)) \\
\nonumber 
&\qquad  \ldots
 \Phi_{x\, {\exp}(\eps V_1)\ldots  {\exp}(\eps^{n_G-1}V_{n_G-1})}{\exp}(\eps^{n_G}V_{n_G}))),
\end{align}

We will use the star product~\eqref{eq_dynkin} from the Dynkin formula and its
properties coming from the gradation property for $H$ via the properties we now describe. 
Here $m\geq 2$ is an integer, which will be equal to $n_G$ below. 
Consider the star product of $m$ elements $W_1,\ldots, W_m$ in $\mathfrak h$ projected onto $\mathfrak h_i$. This product yields a polynomial  map $\mathfrak h^m\to \mathfrak h$
whose linear part 
is the sum $\pr_i (W_1) + \ldots +\pr_i (W_m)$.
The map $P_i=P_{i,m} :\mathfrak h^m\to \mathfrak h$ defined by their difference
$$
P_i (W_1,\ldots, W_m) := \pr_i(W_1*\ldots* W_m) 
- \left(\pr_i (W_1) + \ldots +\pr_i (W_m)\right)
$$
is polynomial, valued in $\mathfrak h_i$ and
 homogeneous in the sense that 
$$
P_i (\delta_r W_1,\ldots,\delta_r W_m)= \delta_r 
P_i (W_1,\ldots, W_m) = r^i P_i (W_1,\ldots, W_m), 
\qquad r>0, \, W_1,\ldots, W_m\in \mathfrak h. 
$$
Furthermore, $P_i$ depends only on the projections of the vectors onto $\mathfrak h_{j'}$, $j'<i$. 
In order to express this technically, for each $j\in \N$, we denote by $\pr_{< j} = \pr_{1} +\ldots + \pr_{j-1}$ the projection onto $\oplus_{j'< j}\mathfrak h_{j'}$ along $\oplus_{j' \geq j} \mathfrak h_{j'}$, with the convention $\pr_{<1}:=0$. We have:
$$
P_i (W_1,\ldots, W_m) = P_i (\pr_{< i} W_1,\ldots, \pr_{< i}W_m). 
$$

Let us come back to~\eqref{eq_PhixV_Vj} and set  $m=n_G$. We use a recursive argument and start with  $i=1$; in this case, $P_1=P_{1,n_G}$ (defined above) is identically zero. Therefore, composing \eqref{eq_PhixV_Vj} with $\eps^{-1} \pr_1 \circ \ln$ yields the expression 
\begin{align*}
\eps^{-1} \pr_1\circ \ln \left(\Phi_x(\delta_\eps{\exp}V)\right)
&=
\eps^{-1}\pr_1\circ \ln
\Phi_x({\exp} (\eps V_1))  
+
\eps^{-1}\pr_1\circ \ln
\Phi_{x\, {\exp} (\eps V_1)}({\exp} (\eps^2 V_2)) 
+
\ldots
\end{align*}
whose limit exists locally uniformly by assumption, see~\eqref{eq_lim_of_epslambdaiVj} for $i=1$. 

\medskip 

The other steps of the recursion are proved in the following manner.
At a general recursive step $i=2,3,\ldots,$
we have
$$
\pr_i \circ  \ln \Phi_x({\exp}(V)) 
=
 Q_{x,i }(V)\ + \ 
 \pr_i\circ \ln
\Phi_x({\exp}( V_1))  +
\pr_i\circ \ln
\Phi_{x\, {\exp}( V_1)}({\exp}(V_2)) 
 +\ldots,
$$
where
$$
Q_{x,i}(V) 
:= 
P_i\left(
\ln\circ \Phi_x({\exp} ( V_1)),
\ln\circ\Phi_{x\, {\exp}  V_1}({\exp} V_2), 
\ldots
\right),
$$
having used the polynomial $P_i=P_{i, n_G}$ defined above.
The properties of $P_i$ imply
$$
\eps^{-i} Q_{x,i }(\delta_\eps V)  =
P_i\left(
\delta_\eps^{-1}\circ\pr_{<i}  \circ \ln \Phi_x( {\exp} ( \delta_\eps V_1)),
\delta_\eps^{-1}\circ \pr_{<i}   \circ \ln\Phi_{x\, {\exp}   \delta_\eps V_1}({\exp} (\delta_\eps V_2)), 
\ldots
\right).
$$
Applying  the recursive assumption to each term involving $\pr_{<i}$, the limit of  
$ \eps^{-i} Q_{x,i} (\delta_\eps V)$ holds locally uniform as $\eps$ goes to $0$. 
  Therefore, in the expression
  \begin{align*}
\eps^{-i}  \pr_i \circ  \ln \Phi_x({\exp}(\delta _\eps V)) 
&=
\eps^{-i} Q_{x,i }(\delta _\eps  V)\ + \ 
 \eps^{-i} \pr_i\circ \ln
\Phi_x({\exp}(\delta _\eps  V_1))  +
\\
&\qquad \eps^{-i}\pr_i\circ \ln
\Phi_{x\, {\exp}(\delta _\eps  V_1)}({\exp}(\delta _\eps  V_2)) 
 +\ldots,
 \end{align*}
 the first term on the right-hand side   has a
locally uniform limit and the other ones too by the hypotheses in  \eqref{eq_lim_of_epslambdaiVj}. This shows the $i^{th}$ step and terminates the proof. 
\end{proof}

\subsection{A refinement on Theorem  \ref{thm_PDiffFP}}

This section is devoted to the statement of the theorem below  which implies Theorem~\ref{thm_PDiffFP} and Corollary~\ref{cor:iso} but is more technical. It will be shown  in Section~\ref{subsec_pfthm:LargeStatement}.

We shall use the following notation:  with a subspace $\mathfrak v$  of $\mathfrak g$ and  an open subset $U$ of $G$, we associate the open set  
$$\Omega_{\mathfrak v,U}=\{(x,V,\eps)\in U\times \mathfrak v\times (0,+\infty) : x\,{\exp_G}(\delta_\eps V)\in U\}$$
and the set
$$\mathcal R_{\mathfrak v,U} = \{(x,V,\eps)\in U\times \mathfrak v\times [0,+\infty) : x\,{\exp_G}(\delta_\eps V)\in U\}.$$

\begin{theorem}\label{thm:LargeStatement}
Let $\Phi$ be a smooth function from an open set $U$ of the graded Lie group $G$ to the graded Lie group $H$.

\smallskip

\noindent \textbf{\rm (i)} 
The map  $\Phi$ is uniformly Pansu differentiable on $U$ if and only if $\Phi$ preserves the filtration at every point $x\in U.$

\smallskip

\noindent \textbf{\rm (ii)} 
Besides, for such a map $\Phi$:
\begin{enumerate}
\item The Pansu derivative yields a smooth function  $(x,z)\mapsto \PD_x\Phi(z)$ on $U\times G$ and the map $x \mapsto \PD_x\Phi$  is smooth from $U$ to $\text{Hom}(G,H)$.
\item For any $x\in U,$   $V\in \mathfrak g_j $,
$$
\pr_{\mathfrak h,i}(\pd_x\Phi (V))
=\lim_{\eps\to 0 }
\eps^{-i}\pr_{\mathfrak h,i} \circ \ln_H (\Phi_x({\exp_G} (\eps^j V ))
=
\left\{\begin{array}{ll} 
0 & \mbox{when}\ j\neq i,\\
\pr_{\mathfrak h,i}(\mathfrak d_x\Phi(V))& \mbox{when}\ j=i.
\end{array}\right. 
$$
Consequently,  the Jacobian $J_\Phi(x)$ equals the Jacobian  of the map $z\mapsto \PD_x(z)$. 
\item For any $(x,V)\in U\times \mathfrak g$ and any $\eps>0$ small enough: 
$$
\delta_\eps^{-1}\left(  \Phi(x)^{-1}  \Phi(x {\exp_G} (\delta_\eps V )) \right)
= 
\exp_H \left( \pd_{x}\Phi(V)+  \eps \rho(x,V,\eps)\right),
$$
where the function $\rho$  
 is  valued in $\mathfrak h$, continuous on  $\mathcal R_{\mathfrak g,U}$ 
 and  smooth on the open subset~
 $\Omega_{\mathfrak g,U}$.
 \end{enumerate}
\end{theorem}

Theorem \ref{thm_PDiffFP} follows from Part \textbf{\rm (i)} while 
the second part of   Corollary \ref{cor:iso} follows from Part \textbf{\rm (i)} and 
Remark \ref{rem_rem:bound}.
The first Part of  Corollary \ref{cor:iso}
is a consequence of Theorem \ref{thm:LargeStatement} (ii)  (1).

Before entering into the proof Theorem~\ref{thm:LargeStatement}, let us give  a technical corollary that will be useful in Section \ref{sec:inv}: 

\begin{corollary}\label{cor:LS}
Here we fix  a homogeneous a quasi-norm $|\cdot|_G$  on $G$  and we keep the same notation  for the function on  the underlying Lie algebra $\mathfrak g$, and similarly for $H$. 
We denote by $\bar B_r$ the corresponding closed balls about 0 of radius $r>0$. 

We continue with the assumptions of Theorem~\ref{thm:LargeStatement}. 
We fix a compact subset $K$ of $U$. 
Let $r_0>0$ be so that 
$K \bar B_{r_0}\subset U$. 
Then there exists a constant $C>0$ such that 
for any $(x,V,\eps)\in K\times \mathfrak g \times (0,1]$ satisfying $\eps|V|_G \leq r_0$, we have:
$$|\rho(x,V,\eps) |_H\leq C (1+ |V|_G)\, \sup_{ x\in K, |W|_G\leq 1, \eps\in[0,r_0]} |\rho(x,W,\eps)|_H.$$
\end{corollary}

\begin{proof}[Proof of Corollary~\ref{cor:LS}]
We set $V=\delta_{|V|_G} V_0$ and we compute
\begin{align*}
\delta_{\eps^{-1}} \ln \circ  \Phi_x ( {\exp_G}\, \delta_\eps V)
&= \delta_{|V|_G}\left( \delta_{\eps^{-1}|V|_G^{-1}}  \ln \circ  \Phi_x ( {\exp_G}\, \delta_{\eps |V|_G} V_0\right)\\
&= \delta_{|V|_G}\left( \pd_{x}\Phi(V_0)  + \eps|V|_G  \rho (x,V_0,\eps|V|_G\right)\\
&=  \pd_{x}\Phi(V)  +\delta_{|V|_G}( \eps |V|_G  \rho (x,V_0,\eps|V|_G)).
\end{align*}
We deduce 
$$\rho(x,V,\eps)= \delta_{|V|_G}( |V|_G  \rho (x,V_0,\eps|V|_G)).
$$
We conclude with the equivalence between homogeneous quasi-norms together with  the observation that if 
$|\cdot|_G$ denotes the homogeneous quasi-norm given by 
$$
|(x_1, \ldots, x_{\dim G})|_G 
= |x_1|^{1/\nu_1} + \ldots+ |x_{\dim G}|^{1/\nu_{\dim G}}
$$
then  for any $t\in\R$ and $W\in\mathfrak g$, $|tW|_{G}\leq (1+|t|) |W|_G$.
\end{proof}

\subsection{Proof of Theorem \ref{thm:LargeStatement}}
\label{subsec_pfthm:LargeStatement}

  \subsubsection{A technical lemma}

The proof of Theorem \ref{thm:LargeStatement} relies on the following property which is of interest on its own, all the more that it holds for any Lie group, not necessarily nilpotent. 

\begin{lemma}
\label{lem_thm_PDiffFP} 
Let $\Phi$ be a smooth function from an open set $U$ of $G$ to $H$.
\begin{enumerate}
\item For every $x\in U$ and $V\in \mathfrak g$,  we have for $t$ in a small neighbourhood of 0,
$$
\partial_t  \ln_H \circ \Phi_x({\exp_G}(tV) ) 
= \sum_{p=0}^\infty\, \widetilde c_{p}  \,\ad^p\left(\ln_H \Phi_x({\exp_G}(tV) )\right) ( \mathfrak d_{x{\exp_G} (tV)} \Phi (V)),
$$
where the coefficients $\widetilde c_{p}\in \R$ comes from the Dynkin formula in \eqref{eq_dynkin}.
In particular $\widetilde c_{0}=1$.
\item 
We have
\begin{align*}
\ln_H \circ\Phi_x({\exp_G}(tV) )|_{t=0}
&=0,
\\
\partial_{t=0}
\ln_H \circ\Phi_x({\exp_G}(tV) ) 
&= \mathfrak d_x \Phi (V),
\\
\partial_{t=0}^2
\ln_H \circ\Phi_x({\exp_G}(tV) ) 
&= \partial_{t=0}\mathfrak d_{x{\exp_G}(tV)} \Phi (V),
\end{align*}
and more generally 
$$
\partial_{t=0}^{k+1}  \ln_H \circ \Phi_x({\exp_G}(tV) ) 
=\partial_{t=0}^{k}   \mathfrak d_{x{\exp_G}(tV)} \Phi (V)+
\sum_{\substack{p,k_2\in \N \\
p+k_2=k}} \binom{k_2}{k} \widetilde c_{p}\, p!\,  \ad^p ( \mathfrak d_{x} \Phi (V))\partial_{t=0}^{k_2}   \mathfrak d_{x{\exp_G}(tV)} \Phi (V) .
$$
 \end{enumerate}
\end{lemma}

\begin{proof}
We observe that 
$$
\partial_t  \ln_H \circ \Phi_x({\exp_G}(tV) )
= 
\partial_{u=0} \ln_H \left(\Phi_x({\exp_G}(tV) )\Phi_{x{\exp_G}(tV)}({\exp_G}( uV) )\right)
= 
\partial_{u=0}  X*Y_u,
$$
where the star product was defined in the introduction (see~\eqref{eq_dynkin})
and
$$
X:=\ln_H \Phi_x({\exp_G}(tV) )
\quad\mbox{and}\quad
Y_u:=\ln_H \Phi_{x{\exp_G}(tV)}({\exp_G} (uV )).
$$
We now apply Dynkin's formula in \eqref{eq_dynkin} and notice that since $Y_0=0$,
only two kinds of terms appear.
Firstly, there are those  with
 $\ell\geq 1$, $s_\ell=1$, $s_1=\cdots =s_{\ell-1}=0$, $r_1>0$, ... $r_{\ell-1}>0$, $r_\ell\geq 0$.
 Secondly there are the ones with 
 $\ell>1$, $s_\ell=0$, $s_{\ell-1}=1$, $s_1=\cdots = s_{\ell -2}=0$, $r_\ell=1$, $r_{\ell-1}\geq 0$, $ r_1,\cdots ,r_{\ell-2}>0$.
 Therefore,  the (finite) sum becomes:
\begin{align*}
\partial_{u=0}  X*Y_u& 
= \sum_{p=0}^\infty \,\widetilde c_{p} \,  \ad^p X (\partial_{u=0}Y_u).
\end{align*}
Using $\partial_{u=0}Y_u=\mathfrak d_{x{\exp_G} tV} \Phi (V)$
proves Part 1. Besides, the term $p=0$ comes from $n=1$, $r_1=0$ and $s_1=1$, whence $\widetilde c_{0}=c_{0,1}=1$. This shows Part (1). 

The first two relations in Part (2) come readily from the definitions of the objects involved. 
This together with the consequence of Part (1)
$$
\partial_t^{k+1}  \ln_H \circ \Phi_x({\exp_G}(tV) ) 
= \sum_{p=0}^\infty \widetilde c_{p} \sum_{k_1+k_2=k} \binom{k_2}{k}
\partial_{t_1=0}^{k_1}\partial_{t_2=0}^{k_2}  
\ad^p\left(\ln_H \Phi_x({\exp_G}(t_1V) )\right) ( \mathfrak d_{x{\exp_G} (t_2V)} \Phi (V)).
$$
implies the rest of  Part (2).
\end{proof}

\subsubsection{Proof of Theorem \ref{thm:LargeStatement}}

Let us start with the equivalence \textbf{\rm (i)}. 

\begin{proof}[Proof of Part \textbf{\rm (i)}  in Theorem \ref{thm:LargeStatement}]
By Lemma \ref{lem_PD_filtration}, it suffices to prove the reverse implication:
we assume that~$\Phi$ preserves the filtration at every $x\in U$ and we want to show that it is uniformly Pansu differentiable on $U$. 
 
Let $V\in \mathfrak g_j$.
By Proposition \ref{prop_MPD}, it suffices to show that 
the limit in \eqref{eq_lim_of_epslambdaiVj} exists and holds locally uniformly. 
Here, we will show the existence of the limits as the local uniformity will be a natural consequence of the existence. 
We only need considering $j<i$ by Lemma \ref{lem_MPD}. 

\smallskip 

As $\Phi$ preserves the filtration at every $x\in U$, by Lemma \ref{lem_eq_filtration_preserving},
we have in a neighborhood of $t=0$ 
\begin{equation}
\label{eq1_pfthm_PDiffFP}
 \pr_{\mathfrak h,i} (\mathfrak d_{x {\exp_G} (tV)} \Phi (V))=0.
\end{equation}
Therefore, differentiating in $t$, we obtain that in the same neighborhood of $t=0$
\begin{equation}\label{toto}
\forall \ell=0,1,2\ldots
\qquad \partial_{t=0}^\ell\pr_{\mathfrak h,i}  ( \mathfrak d_{x {\exp_G} (tV)} \Phi (V))=0.
\end{equation}
More precisely,  by Lemma \ref{lem_eq_filtration_preserving}, we also have 
$$
\forall  p>j \qquad \pr_{\mathfrak h,p} (\mathfrak d_{x {\exp_G} (tV)} \Phi (V))=0,
 $$
and equation~\eqref{eq1_pfthm_PDiffFP} may be generalised into 
\begin{equation}
\label{eq2_pfthm_PDiffFP}
\forall \ell=0,1,2\ldots
\qquad \partial_{t=0}^\ell \mathfrak d_{x {\exp_G} (tV)} \Phi (V) \in \mathfrak h_1\oplus \ldots \oplus \mathfrak h_{j}.
\end{equation}
The latter relation implies that we have 
 for any $p\in\N$ and $\ell =1,2,\ldots$:
$$
\ad^p ( \mathfrak d_{x} \Phi (V))(\partial_{t=0}^{\ell}   \mathfrak d_{x{\exp_G}(tV)} \Phi (V))\in \mathfrak h_1\oplus \ldots \oplus \mathfrak h_{(p+1)j}.
$$
The previous fact together with Lemma \ref{lem_thm_PDiffFP} (2) implies for any $k=2,3,\ldots$
\begin{equation}\label{2Tbis}
 \partial_{t'=0}^{k} \ln_H (\Phi_x({\exp_G} (t' V ))
  \in \mathfrak h_1\oplus \ldots \oplus \mathfrak h_{(k-1)j}.
\end{equation}

\smallskip 

We now  combine these observations with the Taylor expansion of $t\mapsto \pr_{\mathfrak h,i} \circ \ln_H (\Phi_x({\exp_G} (t V )) $ for $t=\eps^j$.
We push the expansion at a higher order depending on how far~$j$ is from $i$. Indeed, for $j<i$, we can find an integer $k\geq 1$ such that $\frac i {k+1} <j< \frac i{k-1}$ (with the convention that $\frac i{k-1}=+\infty$ if $k=1$) and we push the Taylor expansion up to order $k$: 
$$
 \ln_H (\Phi_x({\exp_G} (t V )) = 
 \sum_{k'=1}^k 
 \frac 1{k'!} t^{k'} \partial_{t'=0}^{k'} \ln_H (\Phi_x({\exp_G} (t' V )) 
 +O(|t|^{k+1}). 
$$
By the first part of Lemma \ref{lem_thm_PDiffFP} (2), we have 
\begin{align*}
\ln_H \circ\Phi_x({\exp_G}(tV) )|_{t=0}
&=0,
\;\;
\partial_{t=0}
\ln_H \circ\Phi_x({\exp_G}(tV) ) 
= \mathfrak d_x \Phi (V)=0
\end{align*}
by
\eqref{eq1_pfthm_PDiffFP}.
Therefore,
by \eqref{2Tbis}, 
as long as $(k-1)j<i$, we have
$$
 \pr_{\mathfrak h,i}\circ\ln_H (\Phi_x({\exp_G} (\eps^j V )) = 
 \sum_{k'=1}^k 
 \frac 1{k'!} \eps^{jk'} \pr_{\mathfrak h,i}(\partial_{t'=0}^{k'} \ln_H (\Phi_x({\exp_G} (t' V )) )
 +O(\eps^{j(k+1)})
 =O(\eps^{j(k+1)}) .
$$
 This shows the existence of  the limits in 
\eqref{eq_lim_of_epslambdaiVj} for $j<i/(k-1)$ such that $j>i/(k+1)$, and thus recursively for all $j<i$.
Furthermore, one checks easily that they
hold locally uniformly, so Part~\textbf{\rm (i)} of Theorem~\ref{thm:LargeStatement} is proved.
\end{proof}

\smallskip 

The next points of Theorem \ref{thm:LargeStatement} come from the preceding analysis:

\begin{proof}[Proof of Part \textbf{\rm (ii)}  in Theorem \ref{thm:LargeStatement}]
Additionally to the limits in~\eqref{eq_lim_of_epslambdaiVj}, we have obtained that 
 for any~$V\in \mathfrak g_j$ and $\pr_{\mathfrak h,i}\in\mathfrak h_i$, we have 
\begin{align*}
&\mbox{for} \ i\neq j,\qquad 
\eps^{-i} \pr_{\mathfrak h,i} \circ \ln_H \circ  \Phi_x ( {\exp_G} (\eps^jV))
= 
\eps\, r_{i,j}(x,V,\eps),\\
&
\mbox{for} \ i= j,\qquad 
\eps^{-i} \pr_{\mathfrak h,i} \circ \ln_H \circ  \Phi_x ( {\exp_G}( \eps^jV))
= \pr_{\mathfrak h,i} \circ\mathfrak d_{x}\Phi( V)  
+ \eps\, r_{i,j}(x,V,\eps),
\end{align*}
where the functions $r_{i,j}$ are valued in $\mathfrak h_i$, smooth on the open subset $\Omega_{\mathfrak g_j,U}$ and
continuous on~$\mathcal R_{\mathfrak g_j,U}$.
This implies Point (2) and  that $(x,V)\mapsto \PD_x \Phi({\exp_G} V)$ is smooth on $U\times \mathfrak g_j$ for any $j=1,2,\ldots$
Then, the smoothness of   $\PD_x$ stated in Point (1) follows from  the Baker-Campbell-Hausdorff formula and the fact that it is a group morphism. 
Finally,  the existence of $R$ in Point (3) follows for general $V\in \mathfrak g$,  
using \eqref{eq_PhixV_Vj} and the Baker-Campbell-Hausdorff formula. 
\end{proof}

\section{Invariance of the semi-classical calculus on nilpotent Lie group}\label{sec:inv}

This section is devoted to the proof of  Theorem~\ref{thm:inv}.
We will first recall (see \cite{FF1,FF2}) some properties of the semi-classical calculus 
that will be useful in the proof. 

\subsection{$L^2$-boundedness in the semi-classical calculus}

If $\sigma\in {\mathcal A}_0(U\times \Ghat)$, 
then 
\begin{equation}
\label{eq_L2bdd}
\| {\rm Op}_1 (\sigma) \|_{\mathcal L(L^2(U))}
\leq 
\int_G \sup_{x\in U} |\kappa_x (z)| dz=:\|\sigma\|_{\mathcal A_0(U\times \widehat G)},
\end{equation}
and the right-hand side defines the semi-norm 
$\|\cdot\|_{\mathcal A_0(U\times \widehat G)}$
on $\mathcal A_0(U\times \widehat G)$. 
Consequently,  the operators ${\rm Op}_\eps(\sigma)$ are uniformly  bounded on $L^2(U)$ with bound:
$$
\| {\rm Op}_\eps (\sigma) \|_{\mathcal L(L^2(U))}
\leq 
\int_G \sup_{x\in U} |\kappa_x^{(\eps)} (z)| dz=\int_G \sup_{x\in U} |\kappa_x (z)| dz=\|\sigma\|_{\mathcal A_0(U\times \widehat G)}.
$$
The next lemma shows that the behaviour  of the integral kernel near the diagonal 
$\{(x,x) \ : \ x\in U\}$ in $U$
contains all the information  about ${\rm Op}_\eps(\sigma)$ at leading order.

\begin{definition}\label{def:cutoff}
A \emph{cut-off along the diagonal} of $U$
is a function 
$\chi\in{\mathcal C}^\infty(U\times U) $
such that for each $x\in U$, the map  $z\mapsto \chi (x, xz^{-1})$ extends trivially to a smooth function on $G$ that
is identically equal to 1  on a neighborhood  $U_1$ of $0$ and vanish outside  a bounded neighborhood  $U_0$ of $0$; moreover, the neighborhoods $U_1$ and $U_0$ are assumed to be  independent of $x\in U$. 
We may call the smallest $U_0$ the diagonal support of $\chi$.
\end{definition}
For example, if $\chi_1\in {\mathcal C}_c^\infty(G)$ is a function equal to~$1$ close to~$0$ and with support small enough, then $(x,y)\mapsto \chi_1(xy^{-1})$ is a cut-off along the diagonal of $U$
whose diagonal support is the support of $\chi$.

\begin{lemma}
\label{lem_diag}
Let $\sigma\in \mathcal A_0(U\times \Ghat)$ and let  $\chi\in{\mathcal C}^\infty(U\times U) $ be a cut-off along the diagonal as introduced in Definition~\ref{def:cutoff}.
Let $S_\eps$ be the operator with integral kernel 
$$
(x,y)\mapsto \kappa^{(\eps)}_x(y^{-1}x)\chi(x,y).
$$
Then, for all $N\in\N$, there exists a constant $C=C_{N,\sigma,\chi}>0$ such that 
$$
\forall \eps\in (0,1]\qquad 
\left\|  {\rm Op}_\eps(\sigma)-S^\eps \right\|_{\mathcal L(L^2(U))}
\leq C {\eps}^{N}.
$$
\end{lemma}

\begin{proof}
One observes that  for all $N\in \N$, there exists a bounded function $\theta$ defined on $G\times G$ such that, for all $x,z\in G$
\[
\chi(x,xz^{-1})=1+\theta(x,z) |z|_G^{N}
\]
where $|\cdot|_G$ is a homogeneous quasi-norm on~$G$.
Set $\widetilde \kappa_{x,\eps}(z)= \kappa_x(z) (1-\chi(x,x\delta_\eps z^{-1}))$, then there exists a constant $c_0>0$ such that 
\[
\|\widetilde \kappa_{x,\eps}\|_{\mathcal A_0(U\times G)} \leq c_0 \eps^{N} \int_G\sup_{x\in U} |\kappa_x(z)| \, |z|_G^{N} dz.
\]
This induces the estimate for the operator $S_\eps$.
\end{proof}

\subsection{Proof of Theorem~\ref{thm:inv}}
Let $\sigma\in \mathcal A_0(\Phi(U)\times \Hhat)$
and denote by $\kappa:x\mapsto \kappa_x$ its convolution kernel. 
Set $\widetilde \sigma := (\widehat{\mathbb{G}}\Phi)^*\sigma \in \mathcal A_0(U\times \Ghat)$ and denote by 
$\widetilde \kappa = \PullOne \kappa$ its convolution kernel. 
Note that the $x$-supports are transported via $\Phi$:
$$
x\mhyphen{\rm supp}\, \widetilde\kappa \subset \Phi^{-1} (x\mhyphen{\rm supp}\,  \kappa):=
K_0\subset U.
$$
Let us first proceed to a reduction of the problem by using Lemma~\ref{lem_diag}. 
Let $\chi\in\mathcal C^\infty (\Phi(U)\times \Phi(U))$ be a cut-off along the diagonal of $\Phi(U)$. Consider
the function $\widetilde \chi$ defined on $U\times U$ by 
\[ \widetilde \chi(x,y)= \chi(\Phi(x),\Phi(y)).\]
Then $\widetilde \chi $ is a cut-off along the diagonal of $U$. 

With the two cut-off functions $\chi$ and $\widetilde \chi$ in hands, by Lemma \ref{lem_diag}, we can restrict to proving  that the operator~$R_\eps$ whose integral kernel is 
\begin{align*}
(x,y) \longmapsto \,
&\eps^{-Q} J_\Phi(y) ^{1/2} J_\Phi(x)^{1/2}\kappa_{\Phi(x)}\left(\delta_{\eps^{-1}}(\Phi(y)^{-1}\Phi(x) )\right) 
\chi (\Phi(x),\Phi(y))\\
&\qquad
- \eps^{-Q} \widetilde \kappa_x (\delta_\eps^{-1}(y^{-1} x)) \widetilde \chi(x,y),
\end{align*}
satisfies $\|R_\eps\|_{\mathcal L(L^2(U))} \leq c_0 \,\eps$ for some constant $c_0>0$.
\smallskip

We observe that the operator $R_\eps$ may be written in the form 
$$
R_\eps f(x) = f* r_{\eps,x}^{(\eps)}(x), \quad x\in U, \ f\in \mathcal S(G),
$$
where $r_{\eps,x}^{(\eps)} (y) = \eps^{-Q} r_{\eps,x}(\delta_\eps^{-1}y)$ and 
$r_{\eps,x}$ is the function  in $\mathcal C_c^\infty(U, \mathcal S(G))$ given by
\begin{align*}
r_{\eps,x}(z) &:= J_\Phi(x\delta_\eps z^{-1})^{1/2} J_\Phi(x)^{1/2} \,
\kappa_{\Phi(x) }\left( \delta_{\eps^{-1}}\left(\Phi(x\delta_\eps z^{-1} )^{-1}\Phi(x) \right)\right) \chi \left(\Phi(x),\Phi(x\delta_\eps z^{-1})\right)\\
&\qquad -  \widetilde \kappa_x (z) \widetilde \chi(x,x\delta_\eps z^{-1})\\
&= J_\Phi(x\delta_\eps z^{-1})^{1/2} J_\Phi(x)^{1/2} \,\kappa_{\Phi(x) }\left( \delta_{\eps^{-1}}\left(\Phi(x\delta_\eps z^{-1} )^{-1}\Phi(x) \right)\right) \widetilde \chi(x,x\delta_\eps z^{-1}) \\
&\qquad - J_\Phi(x)\, \kappa_{\Phi(x)}\left( \PD_x\Phi(z)\right) \widetilde \chi(x,x\delta_\eps z^{-1}).
\end{align*}
We now aim to prove that 
\begin{equation}
    \label{eq_aimpfthm:inv}
\exists c_0 >0 \quad \forall \eps\in (0,1]\qquad
I_\eps:=\int_G \sup_{x\in U} 
\left| r_{\eps,x}(z)\right| dz \leq c_0\,\eps,
\end{equation}
as, by \eqref{eq_L2bdd}, this implies $\|R_\eps\|_{\mathcal L(L^2(U))} \leq c_0 \eps$.

We observe that the $x$-support of $r_{\eps,x}$ is included in $K_0$. 
We may assume that the diagonal support of $\chi$ is as small as we need  below and therefore that the diagonal support of $\widetilde \chi$ is included in a small ball $\bar B_{r_0}$ of $G$ with a radius $r_0$ as small as we need; here, we  have fixed a homogeneous quasi-norm  $|\cdot|_G$.
We  can decompose $I_{\eps} \leq I_{1,\eps}+I_{2,\eps}$
with
\begin{align*}
I_{1,\eps}
&:= 
\int_{G} \sup_{x\in K_0} 
\left| J_\Phi(x\delta_\eps z^{-1})^{1/2} -J_\Phi(x)|^{1/2}\right|  |J_\Phi(x)|^{1/2}\left| 
 \kappa_{\Phi(x)}\left( \PD_x\Phi(z)\right)\right| 
 \left|\widetilde \chi(x,x\delta_\eps z^{-1})\right|
 dz,\\
I_{2,\eps}
&:= 
C_2 \int_G \sup_{x\in K_0} 
\left|  
\kappa_{\Phi(x) }\left( \delta_{\eps^{-1}}\left(\Phi(x\delta_\eps z^{-1} )^{-1}\Phi(x) \right)\right)
-  \kappa_{\Phi(x)}\left( \PD_x\Phi(z)\right)\right|
\left|\widetilde \chi(x,x\delta_\eps z^{-1})
 \right|
 dz ,
 \end{align*}
 where 
 $$
 C_2 := \sup_{ K_0} |J_\Phi| ^{1/2} \sup_{K_0 \bar B_{r_0}} |J_\Phi| ^{1/2}.
 $$
 
For $I_{1,\eps}$, we use the Taylor estimates due to Folland and Stein (see~\cite[Section~1.41]{folland+stein_82}  or~\cite[Section~3.1.8]{FR}):
we have for any $x\in K_0$ and $z\in  \delta_{\eps}^{-1} \bar B_{r_0}$
$$
 \left|J_\Phi(x\delta_\eps z^{-1})^{1/2}  -J_\Phi(x)^{1/2}\right|
 \lesssim |\delta_\eps z^{-1}|_G \sup_{x'\in K_0', j=1,\ldots, \dim G} |X_j (J_\Phi^{1/2})(x')|
$$
for some compact subset  $K_0'$ of $U$. Since $y\mapsto  \kappa_{\Phi(x)}\left( \PD_x\Phi(y)\right)\in \mathcal C_c^\infty(U, \mathcal S(G))$, 
this implies that $I_{1} \leq c_1 \,  \eps$ for some constant $c_1>0$. 

\smallskip

For  $I_{2,\eps}$, 
we apply Theorem \ref{thm:LargeStatement} and Corollary \ref{cor:LS} 
with the usual (Euclidean) Taylor estimate to obtain:
\begin{align*}
&\left|\kappa_{\Phi(x) }\left( \delta_{\eps^{-1}}\left(\Phi(x\delta_\eps z^{-1}) ^{-1} \Phi(x)\right)\right)
-
\kappa_{\Phi(x) } (\PD_x\Phi(z))\right|
\\&\qquad=
\left|\kappa_{\Phi(x) }\circ \exp_H \left(
\pd_x\Phi(\ln_G z)  +\eps \rho(x,\ln_G z,\eps)\right)
- 
\kappa_{\Phi(x) }\circ \exp_H  \left(
\pd_x\Phi(\ln_G z)\right)
\right|
\\&\qquad\leq 
\sup_{|W| \leq \eps  |\rho(x,\ln_G z,\eps)|}
\left|D_{\pd_x\Phi (\ln_H z)+W}\kappa_{\Phi(x) }\circ \exp_H \right| \, |\eps \rho(x,\ln_G z,\eps)|
\\
&\qquad\leq D_N (1+|\pd_x\Phi (\ln_G z)|)^{-N} \eps,
\end{align*}
for any $N\in \N$ and some constant $D_N>0$,
as long as $x\in K_0$, $\delta_\eps z^{-1}\in  \bar B_{r_0}$, and $\eps\in (0,1]$, 
since $y\mapsto \kappa_y\in \mathcal C_c^\infty(U, \mathcal S(G))$. 
This implies
$I_2\leq c_2 \eps$
 for some constant $c_2>0$,
 which concludes the proof.

\section{Geometric meaning of the convolution kernel
 }\label{sec:app}

As for the semi-classical calculus in the Euclidean setting, it is important to keep in mind the geometric aspects of the elements that one considers. 
Indeed,  the objects defined using the group Fourier Transform depend on a choice of Haar measure on the graded Lie group $G$. This makes little difference in the group context, where the Haar measure is determined up to a constant, but will matter more if one wants to extend this non-commutative semi-classical approach  on manifolds. In this section, we explain how to avoid choosing a normalization for the Haar measure when defining semi-classical quantization on graded Lie groups. We achieve this by giving a geometrically intrinsic definition of the convolution kernel and thus of the semiclassical symbols.
\smallskip 

We will start with well-known geometric considerations and set out their usual conventions.
The manifolds (often denoted by $M$) are assumed to be smooth. 
If $F = \cup_{x\in M} F_x$ is  a  fiber bundle  over~$M$, 
$\pi_F:F\to M$ will  denote its canonical projection.
Here, all the fiber bundles are smooth, and we denote by $\Gamma(F)$ the space of its smooth sections.  

\subsection{Generalities on  bundles}
Let $E$ be a real vector bundle over a manifold $M$ of rank~$r$. 
\smallskip

For $s\in\mathbb{R}^*$, the {\it $s$-density bundle} associated to~$E$ is the set~$|\Lambda|^s(E)$ of all maps $\mu_x: (E_x)^r:=E_x\times\cdots\times E_x\to\mathbb{R} $, for which
\begin{align*}
    \mu_x(Av_1 \ldots, Av_r)=|\det A|^s\mu(v_1, \ldots, v_r); \quad A\in \text{GL}(E_x).
\end{align*}
When $E=TM$, we simply write $|\Lambda|^sM$
for the  $s$-density bundle over $M$. An $s$-density $\mu$ on $M$ is a smooth section of $|\Lambda|^sM$. When $s=1$ or $s=\frac{1}{2}$, we say that $\mu$ is a {\it density} or {\it half-density} respectively. 
\smallskip 

We say that a density $\mu$ is a {\it positive density} when 
$\mu_x(v_1, \ldots, v_r)>0$ for all $x\in M, \, v_i\in E_x$.
A nonvanishing $n$-form $\omega$ on an open set $U\subset M$ determines a positive density $|\omega|$, and identifies $s$-densities on $U$ with functions via $\mu=f|\omega|^s$. 
In particular, an $s$-density $\mu$ on $M$, is written
\begin{align*}
    \mu(y)=f(y)|dy|^s
\end{align*}
in local coordinates $(y^i)$. 
\smallskip

{\it The canonical $L^2$-space}. Note that two half-densities determine a $1$-density; indeed, given $u=f\mu^\frac{1}{2}$ and $v=g\mu^\frac{1}{2}$, we have $uv=fg\mu$. Furthermore, the integral of a $1$-density is invariantly defined. Therefore we may define the canonical Hilbert space of half-densities,  $L^2(M)$, as the closure of the set of compactly supported half-densities with respect to
\begin{align*}
    \langle u, v\rangle _{L^2(M)}:=\int_M uv.
\end{align*}
Upon choosing a positive density $\mu$ on $M$, we obtain a unitary map between the $L^2$ space associated to $\mu$, and the canonical Hilbert space
\begin{equation}\label{def:Umu}
    U_\mu: L^2(M, \mu)\longrightarrow L^2(M); \quad
    f\longmapsto f\mu^\frac{1}{2}.
\end{equation}

 {\it Densities and mappings}.
The $s$-density bundle, $|\Lambda|^sM$, is a trivial line bundle (though not canonically so), as every smooth manifold has a global positive density. A choice of positive densities~$\mu$ and~$\nu$ on two manifolds~$M$ and~$N$ respectively allows for the definition of the Jacobian of a
smooth map $\Phi: M\to N$. Indeed,
 there exists a positive function $J_\Phi$ on $M$, called the \textit{Jacobian of }$\Phi$, such that 
\begin{align*}
 \Phi^*\nu=J_\Phi\,\mu.
\end{align*}
More generally $\Phi^*\left(f\nu^s \right)=\left(f\circ\Phi \right)J_\Phi^s\, \mu^s.$ 
If in addition, 
$\Phi$ is a diffeomorphism, then
\begin{align*}
    \Phi^*\circ U_\nu = U_{\mu}\circ\, \PullTwo  
\end{align*}
where $\PullTwo $ is defined as in~\eqref{eq:PullTwo} and $U_\mu$, $U_\nu$ in~\eqref{def:Umu}. 
\smallskip

{\it Vertical bundles.}
 Consider a  fiber bundle $F = \cup_{x\in M} F_x$ over a manifold $M$. With $f\in F$, we associate $x=\pi_F(f)$ and $f_x$ the corresponding point in the fiber $F_x$.
 The {\it vertical space} $\mathcal V_f(F)$ above a point $f\in F$ is the tangent space $T_{f_x} F_x$ to the fiber $F_x$ at $f_x$.
 The resulting vector bundle $\mathcal V(F) = \cup_{f\in F}\mathcal V_f(F) $ over $F$ is called the {\it vertical bundle} of $F$; it is  given by the kernel  of the map $d\pi_F: TF\to TM$.
Smooth sections of  $|\Lambda|\mathcal{V}(F)$ are called {\it vertical densities} on $F$. 

\begin{example}[The vertical bundle of a vector bundle]
\label{ex_VE}
Suppose $F=E$ is a vector bundle. Then 
 the vertical space $\mathcal V_e(E)=T_{e_x} E_x$ above   $e\in E$ (here  $x=\pi_E(e)$) is naturally identified with~$E_x$.
 Hence, the vertical bundle $\mathcal{V}(E)$ is isomorphic to  the pull-back  bundle
$(\pi_E)^*(E) = \cup_{e\in E} E_{\pi_E(e)}$ of~$E$
by 
the isomorphism of vector bundles $ \text{vert}^E: (\pi_E)^*\left(E\right)\longrightarrow\mathcal{V}(E)$: 
\begin{align*}
\text{vert}^E_e (V):=\frac{d}{dt}\bigg|_{t=0} U +tV ,
\quad U=e_x\in E_x ,\, V\in E_x \cong T_{e_x} E_x = \mathcal V_e(E), \, x=\pi_E(e). 
\end{align*}
\end{example}

This example may be generalised to the case of a group bundle:

\begin{example}[The vertical bundle of a Lie group bundle]
\label{ex_VG}
Suppose $F=G=\cup_{x\in M} G_x$ is a Lie group bundle. 
The vertical space ${\mathcal V}_{g}(G) = T_{g_x}(G_x)$ above  $g\in G$ (here $x=\pi_G(g)$) is naturally isomorphic to the Lie algebra $\mathfrak g_x$ of the Lie group $G_x$. 
Therefore,  the vertical bundle ${\mathcal V}(G)$ is naturally isomorphic to  the pull-back bundle $\cup_{g\in G} {\mathfrak g}_{\pi_G(g)} = (\pi_G)^* (\mathfrak g)$ of the corresponding Lie algebra bundle $\mathfrak{g}=\cup_{x\in M} \mathfrak{g}_x$. 
More precisely, the natural isomorphism of vector bundles $  \text{vert}^G: (\pi_{G})^*(\mathfrak{g})\longrightarrow \mathcal{V}(G)$ is given via
\begin{align*}
\text{vert}^G_g(V):= \frac{d}{dt}\bigg|_{t=0}g_x\exp_{G_x}(tV),  \quad  V\in \mathfrak{g}_x \cong T_{g_x}(G_x)={\mathcal V}_{g}(G), \, x=\pi_G(g). 
\end{align*}

\end{example}

\subsection{Filtered manifolds}
A {\it filtered manifold} is a  manifold $M$ equipped with a filtration of the tangent bundle $TM$ by vector bundles 
$$
M\times \{0\} = H^0 \subseteq H^1 \subseteq \ldots \subseteq H^{n_M} =TM
\qquad\mbox{satisfying}\qquad 
[\Gamma(H^i),\Gamma(H^j) ]\subseteq \Gamma(H^{i+j});
$$
here we are using the convention that $H^{i} =TM$ for $i>n_M$. 
For each $x\in G$, the quotient
$$
\mathfrak{G}_x M := \oplus_i \left(H^i_x / H^{i-1}_x\right)
$$
is naturally equipped with a structure of graded Lie algebra, 
and we denote by $ \mathbb G_xM$ the corresponding graded Lie group. 
As explained in \cite{VeY1}, the unions
$$
\mathbb{G}M :=\cup_{x\in M} \mathbb G_xM
\qquad \mbox{and}  \qquad
\mathfrak{G}M :=\cup_{x\in M}\mathfrak{G}_x M,
$$
are naturally equipped with a bundle structure that are called  the {\it osculating group and Lie algebra bundles} over $M$.

\begin{example}
Equiregular sub-Riemannian manifolds, in particular contact manifolds, are naturally filtered manifolds.  
Indeed, suppose $M$ is a smooth manifold equipped with a bracket-generating distribution $\mathcal{D}\subset TM$. Let $\Gamma^1=\Gamma(\mathcal{D})$ be the set of smooth sections of $\mathcal{D}$, and 
$$\Gamma^i:=[\Gamma^1, \Gamma^{i-1}]+\Gamma^{i-1}, \quad i>1.$$
As $\mathcal D$ is bracket generating, 
there exists $i$ such that $\Gamma^i = TM$, 
and we denote by $r$ the smallest such integer $i$.
If, for every $i=1, \ldots r$, there exists a subbundle $H^i\subset TM$ for which $\Gamma^i=\Gamma(H^i)$ is the set of smooth sections on $H^i$, then $M$ is said to have an \textit{equiregular} sub-Riemannian structure.  
Clearly, the $H^i$'s provide the structure of filtered manifold. 

 \end{example}
 \begin{remark}
Sub-Riemannian (and semi-Riemannian) structures are not in general equiregular. Examples include the Martinet distribution or the Grushin plane. The latter examples also fall outside of the scope of this paper because they are not open subsets of graded Lie groups. 
%
An approach to a pseudodifferential calculus on the non-equiregular case is given in \cite{LaMoYu}. 
\end{remark}

\medskip 

In what follows, it will be helpful  to describe  a smooth function   $f$   on $\mathbb{G}M$ by denoting  the associated function on the fiber $\mathbb G_x M$ via
$$
\mathbb G_x M \ni z_x\longmapsto f_x(z_x), 
\qquad \mbox{for each}\ x\in M.
$$
We will thus denote by $\mathcal S(\mathbb GM):=\cup_{x\in M} \mathcal S(\mathbb G_xM)$  the (smooth Fr\'echet) bundle of fiberwise Schwartz functions on the group $\mathbb{G}_xM$ over $x\in M$. 
The space
$\Gamma_c(\mathcal{S}(\mathbb{G}M))$  of  compactly supported (smooth) sections of  $\mathcal{S}(\mathbb{G}M)$
 may be described as the space of smooth function~$f$ on~$\mathbb{G}M$
  that are compactly supported in $x\in M$, and for which the functions $z_x\mapsto f_x(z_x)$ are  Schwartz on~$\mathbb{G}_xM$ in a way  varying smoothly with $x\in M$.

\medskip

A {\it morphism of filtered manifolds} is a smooth map $\Phi:M\to N$ between two filtered manifolds~$M$ and~$N$ that respects the filtrations. 
Denoting by $H^{M,i}$ and $H^{N,i}$ the vector bundles giving the filtrations of $TM$ and $TN$ respectively, 
this means that 
$d_x\Phi (H_x^{M,i}) \subseteq H_x^{N,i}$ holds at every $x\in M$ and $i=0, 1, \ldots$ 
At every $x\in M$,  we may define  a
Lie algebra morphism 
$\mathfrak G_x \Phi : \mathfrak G_x M \to \mathfrak G_{\Phi(x)} N$
via
$$
 \mathfrak G_x \Phi (V\, {\rm mod} \ H_x^{M,i-1}) 
 = d_x\Phi (V)\ {\rm mod} \ H_{\Phi(x)}^{N,i-1},
\quad
V\in  H^{M,i}_x, \  i=1, \ldots
$$
We denote by $\mathbb G_x \Phi : \mathbb G_x M \to \mathbb G_{\Phi(x)} N$
the corresponding group morphism. 
This induces \cite{VeY1} morphisms between the osculating  bundles  
$$
\mathfrak G \Phi : \mathfrak G M \to \mathfrak GN
\qquad\mbox{and}\qquad
\mathbb G \Phi : \mathbb G M \to \mathbb GN
$$
that 
we may call the group and (resp.)  Lie algebra {\it osculating maps}.
\smallskip 

Open sets of graded Lie groups are naturally equipped with a structure of filtered manifolds. The next statement summarises 
the main results in  Section \ref{sec:FiltPres}: 
\begin{lemma}
\label{lem_sumsec:FiltPres}
Let $\Phi:U\to H$ be a smooth map from an open set $U$ of a graded group $G$ to a graded group $H$. 
The following are equivalent:
\begin{itemize}
    \item $\Phi$ preserves the filtration in the sense of Definition \ref{def_filtration_preserving}, 
    \item $\Phi$ is locally uniformly Pansu differentiable on $U$ in the sense of Definition \ref{def:pansu},
    \item $\Phi:U\to H$ is a morphism of filtered manifolds (as explained above). 
\end{itemize}
Moreover, in this case, the Pansu derivative and the osculating map coincide at every $x\in U$:
$$
\mathbb G_x \Phi = \PD_x \Phi \, \mbox{on}\ G
\qquad\mbox{and}\qquad
\mathfrak G_x \Phi = \pd_x \Phi \, \mbox{on}\ \mathfrak G.
$$
\end{lemma}

\subsection{Densities on filtered manifolds}
Let $M$ be a filtered manifold. 

\subsubsection{Haar system}
Let us fix  a smooth Haar system $\{\mu_x\}_{x\in M}$ for $\mathbb{G}M$, that is, a choice of Haar measure $\mu_x$ for each of the group fibers $\mathbb{G}_xM$ for which
\begin{align*}
     \forall f\in \mathcal C^\infty(\mathbb{G}M), \quad x\longmapsto \int_{z\in G_xM} f_x(z_x)\mu_x(z_x) \in \mathcal C^\infty(M).
\end{align*} 
A Haar measure $\mu_x$ on the group fiber $\mathbb{G}_xM$ naturally identifies with a left-invariant density, which we also call $\mu_x$, on $\mathbb{G}_xM$. 
By left-invariance, each $\mu_x$ corresponds to a positive element of the one-dimensional vector space $|\Lambda|\mathfrak{G}_xM$. Indeed, smooth Haar systems for $\mathbb{G}M$ correspond in this way to smooth positive sections of $|\Lambda|\mathfrak{G}M$, which we call {\it Haar densities} on $M$. 

\subsubsection{Vertical densities}

As in Examples \ref{ex_VE} and \ref{ex_VG},  the vertical space ${\mathcal V}_z ( {\mathbb G}M) = T_{z_x} {\mathbb G}_xM $ above $z\in {\mathbb G}M$ with $x=\pi_{{\mathbb G}M}(z)\in M$ identifies naturally with the Lie algebra $\mathfrak G_xM$. 
Hence, the vertical bundle ${\mathcal V}({\mathbb G}M) $ is naturally isomorphic to the pullback bundle $ (\pi_{\mathbb{G}M})^* (\mathfrak{G}M )$ of the osculating Lie algebra bundle. 
More precisely, the natural isomorphism of vector bundles 
$$
   \text{vert}: (\pi_{\mathbb{G}M})^* (\mathfrak{G}M )\longrightarrow \mathcal{V}(\mathbb{G}M)
   $$
is given via:
$$
  \text{vert}_z(V) :=  \frac{d}{dt}\bigg|_{t=0}z_x\exp_{\mathbb{G}_xM}(tV), \quad   V\in \mathfrak{G}_xM \cong T_{z_x} {\mathbb G}_M = {\mathcal V}_z ( {\mathbb G}M), \ x= \pi_{\mathbb G M}(z).
$$
This isomorphism allows us to 
lift 
densities $x\mapsto \mu_x$ in $|\Lambda|\mathfrak GM$ to vertical densities $z\mapsto \check\mu_z$ via
\begin{align*}
   \check \mu_{z}\left(\text{vert}_z(V_1), \ldots, \text{vert}_z(V_n)\right)= \mu_x(V_1, \ldots, V_n); \quad V_1, \ldots, V_n\in \mathfrak{G}_xM, \ x=\pi_{\mathbb GM}(z). 
\end{align*}
This lift yields a map of sections:
$     \Gamma\left(|\Lambda|\mathfrak{G}M\right)
    \rightarrow \Gamma\left(|\Lambda|\mathcal{V}(\mathbb{G}M)\right)$.
We shall subsequently blur the distinction between a Haar density on $M$, its lift to a vertical density on $\mathbb{G}M$, and a smooth Haar system on $\mathbb{G}M$.

\subsubsection{Haar system and smooth functions on $\mathbb{G}M$}

With respect to a smooth Haar system $\{\mu_x\}_{x\in M}$ on $M$, 
any  smooth function $\widetilde{\kappa}$ on $\mathbb{G}M$ yields the vertical density 
\begin{align}
    \kappa(z)=\widetilde{\kappa} (z)\mu_{z}, \quad z\in \mathbb{G}M.
    \label{verticalDensity}
\end{align}
Conversely, as $\dim (|\Lambda|\mathfrak{G}_xM) =1$ at every $x\in M$, 
any vertical density $\kappa$ on $\mathbb{G}M$ 
may be written as in \eqref{verticalDensity} for 
a unique smooth function  $\widetilde{\kappa}$  on $\mathbb{G}M$.
This defines an isomorphism of topological vector spaces
$$
    I_\mu: \mathcal {C}^\infty(\mathbb{G}M) \longrightarrow \mathcal{C}^\infty(\mathbb{G}M, |\Lambda|\mathcal{V}); \quad \tilde{\kappa}\longmapsto \kappa =\tilde{\kappa}\mu.
$$

\smallskip

In the next paragraph, we will use 
the (smooth Fr\'{e}chet) bundle
$\mathcal{S}(\mathbb{G}M, |\Lambda|\mathcal{V})$  of fiberwise Schwartz densities; that is, the collection $\{\kappa_x: z_x\mapsto \kappa_x(z_x)\}_{x\in M}$ of densities on the fibers $\mathbb{G}_xM$ for which the function $z_x\mapsto\tilde{\kappa}(z_x)$ is Schwartz.

\subsubsection{Schwartz vertical densities}

A vertical density $\kappa$ is said to be a {\it Schwartz vertical density}  when  the corresponding function $\widetilde \kappa \in C^\infty(\mathbb GM)$ 
from \eqref{verticalDensity} is in $\Gamma_c(\mathcal{S}(\mathbb{G}M))$.
Although this definition requires the choice of a  smooth Haar system $\mu=\{\mu_x\}_{x\in M}$ on $M$,
the resulting property is independent of $\mu$ as the $x$-support  of $ x\mapsto \kappa_x$ is assumed to be compact. 
We denote by $\Gamma_c(\mathcal{S}(\mathbb{G}M, |\Lambda|\mathcal{V}))$
the space of vertical Schwartz densities (compactly supported over $M$). 
Although these spaces are independent of the choice of Haar system, 
fixing such a Haar system  $\mu$
 establishes an identification between 
$\Gamma_c(\mathcal{S}(\mathbb{G}M))$
and $\Gamma_c(\mathcal{S}(\mathbb{G}M, |\Lambda|\mathcal{V}))$
via the isomorphism of topological vector spaces:
\begin{align}
    I_\mu: \Gamma_c(\mathcal{S}(\mathbb{G}M)) \longrightarrow \Gamma_c(\mathcal{S}(\mathbb{G}M, |\Lambda|\mathcal{V})); \quad \tilde{\kappa}\longmapsto \kappa =\tilde{\kappa}\mu.\label{identificationL}
\end{align}

\subsection{Convolution kernels and semi-classical symbols as geometric items}

\subsubsection{Convolution kernels as vertical Schwartz densities}
We start with the following general observation:

\begin{lemma}
\label{lem_presVbundle}
Let $\mathbb{F}: \mathbb{G}M\to \mathbb{G}N$ be a smooth map.
Assume that  $\mathbb{F}_x:\mathbb{G}_xM\to\mathbb{G}_yN$ (with $y=\pi_{\mathbb{G}N}(\mathbb{F}(z_x)$) is a group morphism  for each $x\in M$. Then, the differential of $\mathbb{F}$ preserves the vertical bundles:
$$
d\mathbb{F}\left(\mathcal{V}(\mathbb{G}M)\right) \subseteq \mathcal{V}(\mathbb{G}N).
$$
\end{lemma}
\begin{proof}
For each $z\in \mathbb GM$ and $V\in \mathfrak{G}_xM$ with $x=\pi_{\mathbb GM} (z)$, we have:
\begin{align*}
    d_{z_x}\mathbb{F}_x\left(\text{vert}_z(V)\right)
    &=\frac{d}{dt}\bigg|_{t=0}\mathbb{F}_x\left(z_x\exp_{\mathbb{G}_xM}(tV)\right)
    =\frac{d}{dt}\bigg|_{t=0}\mathbb{F}_x(z_x)\exp_{\mathbb{G}_yN}(t\mathfrak{F}_x(V))\\
    &=\text{vert}_{\mathbb{F}(z)}\left(\mathfrak{F}_x(V)\right),
 \end{align*}
where $\mathfrak{F}_x(V):=\frac{d}{dt}\big|_{t=0}\mathbb{F}_x\left(\exp_{\mathbb{G}_xM}(tV)\right)$.
\end{proof}

Given a map $\mathbb{F}$ as in 
Lemma \ref{lem_presVbundle}, we may now  define the 
 pullback of 
 a vertical density $\nu$ on $\mathbb{G}N$ to a vertical density $\mathbb{F}^*\nu$  on $\mathbb{G}M$ via
\begin{align*}
    \left(\mathbb{F}^*\nu\right)_{z}(V_1, \ldots, V_n):=\nu_{\mathbb{F}(z)}(d_{z_x}\mathbb{F}_x(V_1), \ldots, d_{z_x}\mathbb{F}_x(V_n)); \quad V_1, \ldots, V_n\in T_{z_x}\left(\mathbb{G}_xM\right), \, x=\pi_{\mathbb G M }(z).
\end{align*}
If in addition $\mathbb{F}$ is a diffeomorphism  (and hence an isomorphism of filtered manifold), the pushforward is defined by $\mathbb{F}_*:=(\mathbb{F}^{-1})^*$. 

Given a morphism of filtered manifolds $\Phi: M\to N$, 
we may apply the above to $\mathbb F=\mathbb G\Phi$ the osculating map, 
and define
the pullback  $(\mathbb{G}\Phi)^*\nu$ of a vertical density $\nu$ on $\mathbb{G}N$ to $\mathbb{G}M$.
If $\Phi$ is in addition a diffeomorphism, 
the pushforward yields the following isomorphism of topological vector spaces:
$$
\mathbb G\Phi^*: \Gamma_c\left(\mathcal{S}\left(\mathbb{G}N, |\Lambda|\mathcal{V}\right)\right)\rightarrow \Gamma_c\left(\mathcal{S}\left(\mathbb{G}M, |\Lambda|\mathcal{V}\right)\right).
$$

If $\mu$ and $\nu$ are vertical densities on $M$ and $N$ respectively, we define the operator 
$$
\mathcal{I}_\Phi: \Gamma_c\left(\mathcal{S}(\mathbb{G}N)\right)\to\Gamma_c\left(\mathcal{S}(\mathbb{G}M)\right), \quad \mbox{via}\quad
    \mathbb{G}\Phi^*\circ I_\nu =I_{\mu}\circ \PullOne.
$$

\begin{lemma}\label{lem:geometry}
Assume $M$ and $N$ are open subsets of graded Lie groups~$G$ and~$H$ (resp.). If  $\Phi:M \to N$ is a diffeomorphism that preserves the filtrations of the group, then it is an isomorphism of  filtered manifolds and
the map~$\mathcal{I}_\Phi$ agrees with~\eqref{def:map_kernel}.
\end{lemma}

This result shows that convolution kernels have the geometric structure of vertical Schwartz densities.
\smallskip 

Lemma~\ref{lem:geometry}  motivates $\Gamma_c\left(\mathcal{S}\left(\mathbb{G}M, |\Lambda|\mathcal{V}\right)\right)$ as the geometric space of convolution kernels for future study of semi-classical pseudodifferential operators on filtered manifolds. 
It is also interesting to notice that 
setting 
$ \kappa^{(\eps)}:=\left(\delta_\eps\right)_*\kappa$,
we have for any smooth Haar system, by homogeneity of Haar measures on graded Lie groups, 
\begin{align*}
    \left(\left(\delta_\eps\right)_*\kappa\right)_x(z)=(\delta^*_{\eps^{-1}}\kappa)_x(z)=(\delta_{\eps^{-1}}^*(\widetilde{\kappa}\mu))_x(z)=\widetilde{\kappa}_x(\delta_{\eps^{-1}}z)\delta_{\eps^{-1}}^*\mu_x=\eps^{-Q}\widetilde{\kappa}_x(\delta_{\eps^{-1}}z)\mu_x=:\widetilde{\kappa}_x^{(\eps)}(z)\mu_x,
\end{align*}
where $Q$ is the (constant) homogeneous dimension of the group fibers of $\mathbb{G}M$.

\begin{proof}[Proof of Lemma~\ref{lem:geometry}]
Let $G$ and $H$ be graded Lie groups. 
We may assume that both groups are modeled on $\mathbb{R}^n$, in which case the Lebesgue measure $|dz|$ is a Haar measure for both. If $U$ is an open subset of $G$, then 
$$
\mathcal{V}(\mathbb{G}U)\cong\left(U\times G\right)\times\mathfrak{g}
\qquad\mbox{whence}\qquad 
|\Lambda|\mathcal{V}(\mathbb{G}U)\cong\left( U\times G\right)\times |\Lambda|\mathfrak{g},
$$
and similarly for any open subset of~$H$. With these identifications, the vertical lift of $|dz|$ on $U\subset G$ to $|\Lambda|\mathcal{V}(\mathbb{G}U)$ is also written $|dz|$. By~\eqref{verticalDensity},  any vertical density~$\kappa$ on~$H$ has the form
\begin{align*}
    \kappa_x(z)=\tilde{\kappa}_x(z)|dz|.
\end{align*}
Suppose $\Phi: U\subset G \to \Phi(U)\subset H$ is a filtration-preserving diffeomorphism from the open subset $U\subset G$ to $\Phi(U)\subset H$.  
By Lemma \ref{lem_sumsec:FiltPres}, we have 
$\mathbb G_x \Phi(z)=\PD_x\Phi(z)$ for any $z\in \mathbb{G}_xU=G$, so
\begin{align*}
    \left(\mathbb G\Phi^*\kappa\right)_{x}(z):=\widetilde{\kappa}_{\Phi(x)}(\mathbb{G}_x\Phi( z))\Phi^*|dz|=\tilde{\kappa}_{\Phi(x)}(\PD_x\Phi(z)) J_\Phi(z)|dz|.
\end{align*}
Therefore, vertical Schwartz densities on groups 
transform according to~\eqref{def:map_kernel}. 
\end{proof}

\subsubsection{Towards a set of semi-classical symbols on filtered manifolds}
The {\it fiberwise Fourier transform} of $\kappa\in \Gamma_c\left( \mathcal{S}(\mathbb{G}M, |\Lambda|\mathcal{V})\right)$  is defined as 
\begin{align*}
    \mathcal{F}(\kappa)(x, \pi):=\int_{z\in \mathbb{G}_xM}\kappa(z_x)\pi(z_x)^*; \quad \pi\in \widehat{\mathbb{G}}_xM.
\end{align*}
Here $\widehat{\mathbb{G}}_xM$ is the unitary dual to the group fiber $\mathbb{G}_xM$ (note that this does not require a choice of Haar system).
Then, a natural candidate for the set of {\it semi-classical symbols} on $M$  is the image of $\Gamma_c\left( \mathcal{S}(\mathbb{G}M, |\Lambda|\mathcal{V})\right)$ by $\mathcal{F}$; we denote it by $\mathcal{A}_0(\widehat{\mathbb{G}} M)$. Elements of this space are fields of operators on $\widehat{\mathbb{G}} M:=\cup_{x\in M}\widehat{\mathbb{G}}_xM$ (modulo intertwiners). In the case where $M$ is an open subset of a graded group $G$,  $\mathcal{A}_0(\widehat{\mathbb{G}} M)$ coincides with the set $\mathcal A_0(U\times \widehat G)$ of Section~\ref{subsubsec:sc}.
\smallskip 

Given an isomorphism $\Phi:M\to N$ of filtered manifolds,  the operator defined by
\begin{align*}
    \widehat{\mathbb{G}}\Phi: \widehat{\mathbb{G}}M\longrightarrow\widehat{\mathbb{G}}N; \quad \pi_x\longmapsto \pi_x\circ \mathbb{G}_{\Phi(x)}\Phi^{-1}.
\end{align*}
is a generalization of \eqref{unitaryDualMap}. 
Besides, for $\kappa\in \Gamma_c\left( \mathcal{S}(\mathbb{G}N, |\Lambda|\mathcal{V})\right)$ and $\sigma(x, \pi)=\mathcal{F}(\kappa)(x, \pi)$ we have
\begin{align*}
    \mathcal{F}(\mathbb{G}\Phi^*\kappa)(x, \pi)
    =\widehat{\mathbb{G}}\Phi^* \sigma(x, \pi)
\end{align*}
So when $M$ and $N$ are open subsets of graded Lie groups, we are left with \eqref{unitaryPullBack} and the geometric frame is consistent. 
 
 \bigskip

\section*{Acknowledgements}
 
The authors acknowledge the support of The Leverhulme Trust for this work via Research Project Grant  2020-037, {\it Quantum limits for sub-elliptic operators}. 

We also thank Rapha\"el Ponge for pointing out to us 
his result with Choi Woocheol on Pansu differentiability  on filtered manifolds in 
\cite{Choi+Ponge} - unfortunately after the publication of the present paper to Journal of Geometric Analysis.

\end{document}